\documentclass[11pt]{amsart}
\usepackage{amsfonts,amsmath}
\usepackage{amssymb,latexsym} 
\usepackage{dsfont}
\usepackage{graphicx}

\usepackage{hyperref}
\usepackage{color}

\newtheorem{theorem}{Theorem}[section] 
\newtheorem{lemma}[theorem]{Lemma} 
\newtheorem{proposition}[theorem]{Proposition}

\newtheorem{remark}{Remark}[section]

\numberwithin{equation}{section} 

\newcommand{\K}{\mathcal{K}}
\newcommand{\R}{\mathbb{R}}
\newcommand{\uE}{u^E}
\newcommand{\pE}{p^E}

\newcommand{\ep}{\varepsilon}
\newcommand{\OMep}{\Omega^\ep}
\newcommand{\Aep}{A_\ep}

\newcommand{\RR}{\mathbb{R}}

\newcommand{\pa}{\partial}
\newcommand{\unuep}{u^{\nu,\ep}}
\newcommand{\uep}{u^\ep}
\newcommand{\hep}{h^\ep}
\newcommand{\pnuep}{p^{\nu,\ep}}
\newcommand{\phiep}{\phi^\ep}
\newcommand{\Wnuep}{W^{\nu, \ep}}

\definecolor{Green}{rgb}{0.010,0.7,0.02}

\DeclareMathOperator{\supp}{Supp} 
\DeclareMathOperator{\dive}{div}
\renewcommand{\div}{\dive}

\DeclareMathOperator{\curl}{curl} 

\title[Vanishing viscosity in porous medium]{The vanishing viscosity limit in the presence of a porous medium} 
\author[C. Lacave]{Christophe Lacave} 
\author[A. Mazzucato]{Anna Mazzucato}

\address[C. Lacave]{Univ Paris Diderot, Sorbonne Paris Cit\'e, Institut de Math\'ematiques de Jussieu-Paris Rive Gauche, UMR 7586, CNRS, Sorbonne Universit\'es, UPMC Univ Paris 06, F-75013, Paris, France.} 
\email{christophe.lacave@imj-prg.fr}

\address[A. Mazzucato]{Penn State University, University Park, PA 16802, U.S.A.}
\email{alm24@psu.edu}
\thanks{The first author is partially supported by the Agence Nationale de la Recherche, Project DYFICOLTI grant ANR-13-BS01-0003-01, and  by the project \emph{Instabilities in Hydrodynamics} funded by the Paris city hall (program \emph{Emergences}) and the Fondation Sciences Math\'ematiques de Paris.\\ 
The second author was partially supported by the U.S. National Science 
Foundation grants DMS-1009713, DMS-1009714, and DMS-1312727.}

\date{\today}

\begin{document}

\begin{abstract}
We consider the flow of a viscous, incompressible, Newtonian fluid in a 
perforated domain in the plane. The domain is the exterior of a regular lattice 
of rigid particles. We study the simultaneous limit of vanishing particle size
and
distance, and of vanishing viscosity.
Under suitable conditions on the particle size, particle
distance, and viscosity,  we prove that solutions of the Navier-Stokes
system in the perforated domain converges to solutions of the Euler system,
modeling inviscid, incompressible flow, in the full plane. That is, the flow is 
not disturbed by the porous medium and becomes inviscid in the limit. 
Convergence is obtained in 
the energy norm with explicit rates of convergence.
\end{abstract}


\maketitle

\section{Introduction} \label{sec:intro}

This article concerns the vanishing viscosity limit for incompressible, 
Newtonian fluids in a perforated planar domain when the size and distance 
between obstacles vanish. By a perforated domain, we mean the 
exterior of a (finite) regular lattice of rigid particles.
This problem can be viewed as simultaneously taking the limit of vanishing 
viscosity for flows in exterior domains and  study the permeability of a
porous medium under this flow.

From a physical point of view, it is important to study fluid flow through a 
porous medium in the regime of small viscosity, which model for example flow of 
underground water. Away from boundaries, the effect of viscosity is 
negligible generically for such flows, but near walls, viscous 
boundary layers appear where large stresses and vorticity can develop.  The 
layer can in turn destabilize, which lead to boundary layer separation and the 
formation of a turbulent wake. (We refer to \cite{Schlichting} for an 
introduction to the theory of boundary layers.)  

We say that the {\em vanishing viscosity limit} holds on the time interval 
$[0,T]$ if the solutions of the Navier-Stokes equations converge to the solution of the 
Euler equations strongly in the energy norm, i.\thinspace e., strongly in 
$L^\infty([0,T];L^2)$.  
It is not known whether the vanishing viscosity limit holds  in 
domains with boundaries, not 
even in two space dimensions, if no-slip boundary conditions are imposed on 
the velocity. The 
Reynolds number is a dimensionless quantity, which is inversely proportional to 
the viscosity. When the characteristic local length scale near 
the boundary becomes small at a faster rate than the viscosity, one can expect 
that the local Reynolds number stays of order one, preventing the formation of 
a strong boundary layer, as was observed in \cite{ILN09}, where the vanishing 
viscosity limit was established in the case of one shrinking obstacle. 
Homogenizing the Navier-Stokes equations in the perforated (periodic) domain 
with Reynolds number below a critical value yields Darcy's law as in the case 
of the Stoke's system, while homogenizing the Euler equations gives non-linear 
filtration laws that depend on the relation between the particle size and the
characteristic velocity of the problem. In the intermediate regime of high 
Reynolds numbers, asymptotic analysis leads to consider the Prandtl boundary 
layer equations, which have recently been shown to be ill-posed 
\cite{GVD10,GVN12}, for the cell problem.  (We refer to \cite{MikelicPaoli} and 
references therein, in particular \cite{Mikelic95}, for a more detailed 
discussion. See also \cite{BMpreprint} for a recent result on stochastic
homogenization.) 
In the case of the Darcy-Brinkman system, the equations for 
the boundary corrector are linear and passage to the zero-viscosity limit is 
possible \cite{KTW11,HW14}.

It is therefore important to identify regimes where it is possible 
to neglect the effects of both the viscous boundary layers as well as the 
porosity of the medium and use the Euler equations in the full plane. In this 
article, we establish the convergence  of solutions to the two-dimensional
Navier-Stokes equations in the perforated domains to solutions of the Euler
equations in the full plane  in the energy norm,  when viscosity, size of
the particle, and particle distance vanish in the appropriate regime: namely,
particle
size must be smaller than particle distance, and a certain ratio of  particle
size to particle distance must be bounded above by viscosity. We
consider the most difficult and interesting case of no-slip boundary conditions 
for the Navier-Stokes problem. 
We confine ourselves to flows in the plane for technical reasons. However, our 
methods could be adapted to the more challenging study of 3D flows, at least
for the case of a fixed number of obstacles (see the related discussion in 
\cite{ILN09}).

The study of 2D incompressible flows in the exterior of  a 
single shrinking obstacle, both for viscous and inviscid fluids, was carried 
out in the series of papers \cite{IK09, ILN03,ILN06,ILN09} (see also
\cite{CPRXpreprint} for results in the periodic setting), where the case
of obstacles diametrically shrinking to a point was considered. This is the 
geometric set up in our paper as well.  

Even in the case of flows in the exterior of a finite number of 
shrinking obstacles at fixed positions, extending the results in 
\cite{ILN09} is not straightforward, as one needs to choose an approximation to 
the Euler solution  which is divergence free and  satisfy the no-slip condition at the boundary. For one 
obstacle, it is sufficient to truncate the stream function associated to the
Euler velocity. For more than one obstacle, we gain a good control on the norm
of the 
approximation by truncating the velocity and by constructing instead an 
appropriate corrector to restore the divergence-free condition. When the 
obstacles are kept at a fixed distance from each other, our result applies to 
arbitrary configurations of obstacles.

When the problem is studied in a porous medium, additional 
difficulties arise. In fact, the permeability of the medium in the limit depends
on the relation between particle distance and particle size. This relation
changes
drastically from the viscous to the inviscid case (cf.
\cite{Allaire90a,Allaire90b} 
for the viscous case, and \cite{BLM,LM,LionsMasmoudi,MikelicPaoli} for the inviscid case). 
Therefore, a careful analysis of the flow in the perforated domain is needed to 
pass to the limit, in particular with respect to the choice of initial data for 
the Navier-Stokes system. In the regime considered in this paper, we  take 
the initial data as well as the Euler correctors to be independent of
viscosity. We will expand on
this point further in Subsection~\ref{sec:domain}.

We close with an outline of the paper.
In the rest of the introduction we describe in detail the perforated domain and 
recall the fluid equations: the Navier-Stokes equations modeling viscous flows, 
and the Euler equations modeling inviscid flows.

We will compare the Navier-Stokes equations in the exterior of the obstacles, 
extended by zero, to the Euler solution in the whole plane. To this end, we 
will construct suitable approximations of the Euler solution in the exterior 
domain that satisfy the no-slip condition at the boundary and are still
divergence free.
In the case of one obstacle, the authors in \cite{ILN09} truncate the stream 
function, as both the stream function and the pressure are defined up to 
an arbitrary constant, and can then be chosen to be small on the obstacle. 
Smallness of the stream function and the pressure avoid a sub-optimal estimate 
on the gradient of the truncation. This construction does not extend to more 
than one obstacle. In Section~\ref{sec:prelim}, we present a different 
construction, which utilizes a direct cut-off function on the velocity and a
correction to restore the divergence-free condition.
This construction was introduced in \cite{Lac-NS3D} and is based on
Bogovski{\u\i} operator.
Section~\ref{sec:prelim} also contains the elliptic estimates that are needed
to establish our main convergence result. We also recall some
standard estimates for the Euler solutions.

In Section~\ref{sec:initial}, we discuss the convergence of the Navier-Stokes 
initial data to the Euler initial data under the conditions on 
$\ep$ (the particle size) and $d_\ep$ (the distance between particles)
given in Theorem~\ref{th:main1}, using the analysis in
\cite{BLM}. 

Finally, in Section~\ref{sec:energy}, we prove convergence of the Navier-Stokes 
solutions $\unuep$ to the Euler solution  $\uE$ by correcting the  
Euler solution thanks to  the results 
in Section~\ref{sec:prelim}. In 
fact, we prove a more general stability result--- see the statement of
Theorem~\ref{th:main}--- analogous to the main result in \cite{ILN09}. The
convergence
follows from this stability result and the convergence of the Navier-Stokes initial 
data to the Euler data, discussed in Section~\ref{sec:initial}.

Throughout the paper $\nabla$ will denote the gradient of a function with 
respect to its argument, while $C$ will denote a generic constant that may 
change from line to line. We will represent a point in $\RR^2$ by $x=(x_1,x_2)$ 
and denote the ball in $\RR^2$ (disk) with center $x$ and radius $R$ by 
$B(x,R)$.

\subsection*{Acknowledgments} 
A.~M. would like to acknowledge the hospitality and financial support of 
Universit\'e Paris-Diderot (Paris 7), Universit\'e Pierre et Marie Curie (Paris 
6), and the financial support of Paris City Hall and the Fondation Sciences
Math\'ematiques de Paris, to
conduct part of this work.

The authors are also grateful to the anonymous referee for his valuable comments on the first version of this article.

\subsection{The perforated domain} \label{sec:domain}

We begin by describing in detail the geometric setup. By a perforated domain we mean the complement of finitely-many inclusions arranged in a lattice.
In this article, we consider the case where all the inclusions have the same shape and the lattice is regular:
\begin{equation}\label{Kijeps}
\K_{i,j}^\varepsilon := z_{i,j}^\varepsilon + \varepsilon \K
\end{equation}
where
\begin{equation}\label{K}
\begin{aligned}
\K \text{ is a connected, simply connected compact set of }\RR^2,\\
 \K \subset [-1,1]^2 ,\quad \partial \K \in C^{1,1}\quad \text{ and }\quad  0\in \stackrel{\circ}{\K}.\end{aligned}
\end{equation}
The regularity of the lattice is needed only when the distance
between the inclusions vanishes. If a fixed number of shrinking obstacles is
considered instead, these can be arranged in an arbitrary fashion.

In \eqref{Kijeps}, $z_{i,j}^\varepsilon$ should be chosen such that the inclusions are disjoints. We will assume that the inclusions are at least separated by a distance $2d_{\varepsilon}$:
\begin{equation}\label{zij}\begin{split}
z_{i,j}^\varepsilon :=& (\varepsilon+2(i-1)(\varepsilon+d_{\varepsilon}),\varepsilon+2(j-1)(\varepsilon+d_{\varepsilon}))\\
=&(\varepsilon,\varepsilon)+2(\varepsilon+d_{\varepsilon}) (i-1,j-1),
\end{split}\end{equation}
see Figure \ref{fig.config}.
We are interested in the case for which $d_{\varepsilon}$ goes to zero as
$\varepsilon\to 0$,  hence the case for which the number of inclusions goes to
infinity.

In the horizontal direction, we distribute the inclusions on the unit segment, i.e. we consider
\[
i = 1,\dots, n_{1}^\varepsilon \quad \text{with}\quad  n_{1}^\varepsilon=N^{\varepsilon}:= \Big[ \frac{1+2d_{\varepsilon}}{2(\varepsilon+d_{\varepsilon})}\Big]
\]
where $[x]$ denotes the integer part of $x$. In the vertical direction, 
we characterize the geometry by introducing a parameter $0\leq \mu\leq 1$. The case $\mu=0$ corresponds to the case in which we only
have obstacles arranged on a line, whereas the case $\mu=1$ corresponds to the case 
in which we have obstacles arranged in a square lattice.
This configuration is achieved by considering  $[(N^{\varepsilon})^\mu]$
horizontal lines of obstacles, that is,
\[
j = 1,\dots, n_{2}^{\varepsilon}\quad \text{with} \quad n_{2}^\varepsilon := \Big[ (N^{\varepsilon})^\mu\Big].
\]
\begin{figure}[h!t]
\begin{center}
\includegraphics[height=6cm]{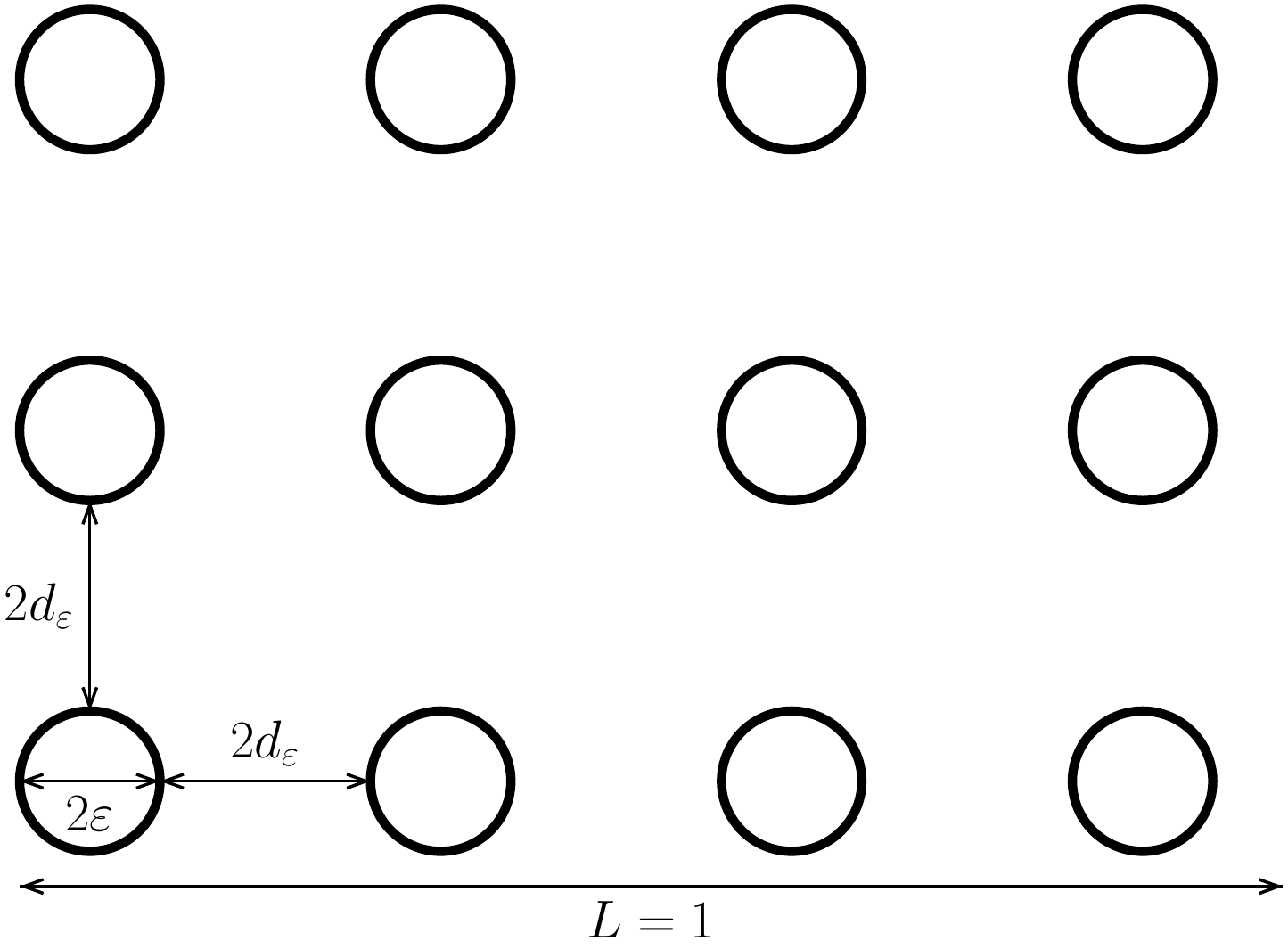}
\caption{Distribution of inclusions.}\label{fig.config}
\end{center}
\end{figure}
The intermediate case $0<\mu<1$ corresponds to a rectangular configuration of
inclusions, such that in the limit $\ep\to 0$ the density of particles becomes
infinite, but the height of the rectangle vanishes.
For simplicity, we carry out the analysis only in the case that all the
inclusions have the same shape $\K$, but similar results can be established when
obstacles of different shapes shrink homotetically to zero. When the number of
inclusions becomes infinite, we need to assume that the possible shapes are
finitely many, in order to have uniform constants in elliptic estimates (see
Subsection~\ref{sec:elliptic}).
The main reason for studying different arrangements of inclusions as 
$\varepsilon$ goes to zero is that the limiting equations may change depending 
on the geometric configuration (this is indeed the case for the viscous 
homogenization of the perforated domain \cite{Allaire90a,Allaire90b}).

Throughout the paper, the fluid domain will be the exterior of the obstacles:
\begin{equation}\label{OM}
\Omega^{\ep,\mu} = \Omega^\varepsilon:=\RR^2\ \setminus 
\Big(\bigcup_{i=1}^{n_{1}^\varepsilon} 
\bigcup_{j=1}^{n_{2}^\varepsilon} \K_{i,j}^\varepsilon\Big).
\end{equation} 
For notational convenience, we suppress the dependence of the fluid domain on 
$\mu$.

{\em The main purpose of this paper is to study the simultaneous limit of 
vanishing viscosity
$\nu$ and the limit  of vanishing $\varepsilon$ for the solutions of the
Navier-Stokes 
equations in the exterior domain $\Omega^\varepsilon$.
For different values of $\mu$, we
determine a relation between $\nu$, $\varepsilon$, and $d_\ep$ such that
the
Navier-Stokes solution in the perforated domain converges to the solution of the
Euler equations in the 
full plane.}

The asymptotic of fluids outside obstacles that shrink to points is a very
active area of research. For inviscid flows, Iftimie, Lopes Filho, and
Nussenzveig Lopes have treated in \cite{ILN03}  the case of one shrinking 
obstacle in the plane (extended to the case of one shrinking obstacle in a non 
simply-connected bounded domain by Lopes Filho in \cite{Lopes07}), whereas Lacave, 
Lopes-Filho, and Nussenzveig-Lopes have studied in \cite{LLL} the case of an 
infinite number of shrinking obstacles. These 
works consider also the case where the circulation of the initial velocity around $\K_{i,j}^\ep$ 
 is non zero, but do not provide a control on the distance between
the holes. In the case of zero circulation (see \eqref{u0eps}), Bonnaillie-No\"el, 
Lacave, and Masmoudi prove in \cite{BLM} that, when $d_\ep =\ep^\alpha$, solutions of the Euler
equations in the exterior domain $\Omega^\varepsilon$ defined in \eqref{OM}
converge to the Euler solution in the full plane if
$\alpha<1$ (more precisely, if $\alpha<2-\mu$).
Hence, in the limit the ideal fluid is not perturbed by the perforated domain
if $d_{\varepsilon} = \varepsilon^\alpha$ with $\alpha<1$. For larger $\alpha$, 
Lacave and Masmoudi establish in \cite{LM} that the perforated domain becomes 
impermeable (e.g. when $\mu=1$, an impermeable square appears if  $\alpha>1$). 
In the periodic setting   (related 
to the case $\mu=\alpha=1$ in the context of this paper), an homogenized limit 
for a modified Euler system was obtained in  
\cite{LionsMasmoudi,MikelicPaoli}.

Concerning viscous fluid with fixed viscosity $\nu$, Iftimie, Lopes-Filho, and
Nussenzveig-Lopes  considered again the case of one shrinking
inclusion in \cite{ILN06} (see also \cite{IK09} in dimension three). There is
also an extensive literature in the homogenization framework. Following the
pioneering work of 
Cioranescu and Murat \cite{CM82} for Laplace's equation, Allaire  studied the
homogenization of the Navier-Stokes equations in a perforated domain under 
different regimes, given in terms of relative ratios of $d_\ep$ to $\ep$, in
\cite{Allaire90a,Allaire90b}. When $\mu=1$ and  $d_{\varepsilon} = 1/\sqrt{|\ln 
\ep|}$, or when $\mu=0$ and  $d_{\varepsilon} = 1/ {|\ln \ep|}$  for 
the critical distance between inclusions,  the limit problem is 
described by a filtration law of Brinkman type. If the distance is larger than 
this critical value, one recovers the solution of the Navier-Stokes in the full 
plane without any influence of the porous medium, and if the distance is 
smaller, one obtains Darcy's law in the limit.
We refer to \cite{Allaire90b,BLM} and references therein for a more detailed 
discussion of this point, in particular for results on viscous flow through a 
sieve.
We remark that the dependence of the critical distance on $\ep$ needed to feel
the presence of the porous medium in the limit is markedly different in the 
viscous and inviscid cases. For example, for distances $d_\ep$ 
satisfying  $\varepsilon^\alpha \ll d_{\varepsilon} \ll 1/ \sqrt{|\ln \ep|}$, 
the limiting system describing the homogenized viscous flow is different than
the Navier-Stokes system, whereas there is no memory of the obstacles in the
limit for ideal flows. As a matter of fact, we are able to treat the situation in which the distance is smaller than
the critical distance (namely, $d_{\ep}^{1/2} \sqrt{|\ln
\varepsilon|} \to 0$) and the inclusions are distributed along a line, i.e., $\mu=0$ in our setting, albeit with a less than optimal rate of convergence.  For Navier-Stokes, this case was not addressed by Allaire in \cite{Allaire90b} (see Remark~\ref{rem:end}).

It is therefore natural to investigate the concurrent
limit $\ep,\nu\to 0$.
There are also important physical motivations for studying this problem, as
discussed in the introduction. (We again refer to \cite{MikelicPaoli} 
and \cite{ILN09} for a more in-depth discussion of the physical 
implications of these and our results.)

\subsection{Fluids equations and initial data}

We consider the flow of a viscous, incompressible, Newtonian fluid in 
$\Omega^\varepsilon$ with classical no-slip boundary conditions on
$\pa\Omega^\varepsilon:= \cup_{i,j}  \pa \K_{i,j}^\ep$.
We hence let $(\unuep, \pnuep)$ denote the solution of the Navier-Stokes 
equations with (kinematic) viscosity coefficient $\nu$ in $\Omega^\varepsilon$ 
and initial velocity $u_{0}^{\ep}$:
\begin{equation} \label{eq:NSE}
  \begin{cases}
       \unuep_t   - \nu \Delta \unuep +( \unuep\cdot \nabla) \unuep +\nabla 
\pnuep = 0, &
 \text{on } (0,+\infty)\times \Omega^\varepsilon, \\
  \dive \unuep =0, & \text{on } [0,+\infty)\times \Omega^\varepsilon,\\
\unuep  = 0, & \text{on } (0,+\infty) \times\pa\Omega^\varepsilon , \\
 \unuep(0,\cdot) =u_0^{\ep}, & \text{on } \Omega^\varepsilon.
  \end{cases}
\end{equation}
We take initial data depending only on $\ep$ and not directly on viscosity
$\nu$ for reasons that will be clear below. However, this set up can be easily 
relaxed. The initial data is assumed only to  satisfy the no-penetration
conditions forced by the impermeable boundaries.
By abuse of notation, we will often refer to the velocity $\unuep$ as the 
Navier-Stokes solution, without mentioning the pressure $\pnuep$.

We will also consider inviscid flow in the whole plane $\RR^2$, and denote by 
$\uE$ the solutions of the Euler equations with initial data $u_0$:
\begin{equation} \label{eq:Euler}
  \begin{cases}
       \uE_t   + (\uE\cdot \nabla) \uE +\nabla \pE = 0, &
 \text{on } (0,+\infty)\times \R^2, \\
  \dive \uE =0, & \text{on } [0,+\infty)\times  \R^2,\\
 \uE(0,\cdot) =u_0, & \text{on }  \R^2.
  \end{cases}
\end{equation}

Our main result is a convergence result of $\unuep$ to $\uE$ under  suitable conditions on the relative strengths of $\nu$, $\varepsilon$, 
and $d_\varepsilon$, assuming convergence of the respective initial data. 
Such a convergence can be achieved if $u_0^\ep$ is chosen to be an appropriate 
truncation of the Euler initial velocity $u_0$ that is tangent to the
boundary of $\Omega^\ep$ and divergence free. $u_0^\ep$ will be close to $u_0$ 
owing to the fact that the size of the obstacles is small.
A standard approach to constructing the truncated velocity is to give the 
initial Euler velocity $u_0$ in term of a fixed initial vorticity 
independent of the domain (as e.g. in 
\cite{BLM,IK09,ILN03,ILN06,ILN09,Lac-NS3D,LLL,LM}). More precisely, let  
$\omega_{0}$ be a smooth function compactly supported outside the porous medium (i.e. compactly supported in $\R^2 \setminus ([0,1]\times \{ 0\})$ if $\mu \in [0,1)$ and in $\R^2 \setminus ([0,1]^2)$ if $\mu =1$), and define 
$u_0$ by
\begin{equation}\label{u0}
u_{0}(x) = K_{\R^2}[\omega_{0}](x):= \frac1{2\pi}\int_{\R^2} 
\frac{(x-y)^\perp}{|x-y|^2} \omega_{0}(y)\, dy,
\end{equation}
where $K_{\RR^2}$ denotes the Biot-Savart operator in the full plane.
That is, $u_0$ 
is the unique solution of the problem
\begin{equation}\label{eq:u0}
\div u_{0} =0 \text{ in } \R^2,\quad \curl u_{0} =\omega_{0} \text{ in } 
\R^2,\quad \lim_{x\to \infty} |u_{0}(x)|=0.
\end{equation}
It is well known (see e.g. \cite{MajdaBertozzi}) that, with this initial data, 
the Euler system~\eqref{eq:Euler} has a unique solution $\uE$.  The properties
of such a solution will be recalled in Subsection~\ref{sec:Euler}

\begin{remark}
The condition that $\omega_0$ be supported outside of the obstacles is used
mainly in proving convergence of the initial velocities in $L^2$ in
Section \ref{sec:initial}, which is then utilized in the energy estimates of
Section \ref{sec:energy}. In fact, to show that $u^{\nu,\ep}_0-u_0$ belongs to
$L^2(\Omega^\ep)$ one needs
\[
   \int_{\Omega^\ep} \omega_0\, dx= \int_{\RR^2} \omega_0\, dx,
\]
and the condition on the support of $\omega_0$ guarantees that this condition
holds uniformly in $\ep$. By comparison, in the different context of flow in a
half-plane, Maekawa \cite{Mae} recently proved that the
vanishing vioscosity limit holds, if the initial
vorticity is (compactly) supported away from the boundary, for a time that is at
least comparable to the distance of the support of vorticity from the boundary.
Note, however, that in Theorem \ref{th:main1}, the vanishing viscosity limit is
established for arbitrarily long intervals of time over which we cannot exclude
that the support of the Euler vorticity intersects the obstacles.
\end{remark}

For any $\varepsilon>0$, there exists a unique vector field $u_{0}^\varepsilon$
defined on $\Omega^\varepsilon$ verifying:
\begin{equation}\label{u0eps}\begin{split}
\div u_{0}^\varepsilon =0 \text{ in } \Omega^\varepsilon,\quad \curl 
u_{0}^\varepsilon =\omega_{0} \text{ in } \Omega^\varepsilon, \quad  
u_{0}^\varepsilon \cdot n =0 \text{ on } \partial \Omega^\varepsilon, \\
  \int_{\partial \K_{i,j}^\varepsilon}u_{0}^\varepsilon\cdot \tau\, ds =0 \text{ 
for all } i,j,\quad \lim_{x\to \infty} |u_{0}^\varepsilon(x)|=0,
\end{split}\end{equation}
see \cite{Kikuchi}. The convergence of $u_0^\ep$ to $u_0$ will be 
established in Section~\ref{sec:initial} using the methods of \cite{BLM}.
As we will show in \eqref{u0ep regularity}, the vector field 
$u_{0}^\varepsilon$ constructed in \eqref{u0eps} belongs to $L^\infty\cap 
L^{2,\infty}(\Omega^\varepsilon)$. Hence, we can apply the result of Kozono and 
Yamazaki \cite[Theo. 4]{KY95} to conclude that the problem \eqref{eq:NSE} has a
unique, global-in-time, strong solution. 

We close by observing that neither the Euler nor the Navier-Stokes solution 
have finite energy, nevertheless we will be able to show convergence, as $\nu$,
$\ep$ vanish,  of $\unuep -\uE$ to zero in the energy norm.

\subsection{Main result} 

In this work, we are interested in studying the limit of vanishing viscosity
$\nu\to 0$ at
the same time as the limit $\ep\to 0$, in which the obstacles shrink to points.
In the case of a single obstacle,  it was proved in \cite{ILN09} that, for any 
given $0<T<\infty$, there exists a constant 
$C_{1}=C_{1}(\K,u_{0},T)$ such that, under the conditions $\ep \leq C_{1}\nu$,  
$\unuep$ converges to $u^E$ in the energy space $L^\infty((0,T),L^2(\RR^2))$. In 
this result and in our main theorem, $\unuep$  and $u_0^\ep$ are extended by 
zero to the whole plane. Unfortunately, the arguments in \cite{ILN09} use in a 
crucial way the fact that we have only one obstacle: namely, the authors utilize
the
stream function $\psi^E$ and choose a pressure $\pE$ for the Euler system that
is zero at the center of $\K$, in order to construct a corrector for the
Navier-Stokes solution in the fluid domain via truncation. Taking $\psi^E$ and 
$p^E$ small as the obstacle shrinks is needed to balance the growth of 
gradients of the cut-off function. Such a choice is not possible in the case of 
more than one obstacle, as there is only one degree of freedom for the Euler 
pressure and stream function in the whole plane. Inspired from a method 
developed in \cite{Lac-NS3D}, we will instead directly truncate the velocity 
field with good estimates on the truncation, which will allow to adapt the 
energy argument of \cite{ILN09} to establish the zero-viscosity limit. 

Due to the inherent difficulty with the vanishing viscosity in bounded domains, 
we are only treating the case where $d_{\varepsilon}\geq \varepsilon$ (without 
loss of generality we can assume that $d_\ep\leq 1$), hence the 
total number of inclusions satisfies 
\begin{equation}\label{n1n2}
 n_{1}^\varepsilon n_{2}^\varepsilon \leq \frac1{d_{\varepsilon}^{1+\mu}}.
\end{equation}

Our main result is the following theorem.

\begin{theorem}\label{th:main1}
Given $\omega_{0}\in C^\infty_{c}(\R^2)$,  let $\uE$ the solution of the Euler 
equations \eqref{eq:Euler} in the whole plane with initial condition $u_{0}$ 
(given in terms of $\omega_{0}$ in \eqref{u0}). For any $\ep,\nu>0$, let $d^\ep 
\geq \ep$ and let $\Omega^\ep$ be defined in \eqref{Kijeps}-\eqref{OM}. 
Let $\unuep$ be the solution of  the Navier-Stokes equations \eqref{eq:NSE} in 
$\Omega^\ep$ with initial velocity $u_{0}^{\ep}$, given by the unique solution 
of \eqref{u0eps}).
Then, there exists a constant $A$ depending only on  $\K$ such that if
$$\displaystyle \frac{\ep}{d_{\varepsilon}^{(1+\mu)/2}} \leq \frac{A \nu}{ \| \omega_{0}\|_{L^1\cap L^\infty(\R^2)}},$$
and if $\omega_{0}$ is supported in $\Omega^\varepsilon$, then for any $T>0$ we have
\begin{equation} \label{ineq.main}
         \sup_{0\leq t\leq T}  
  \|\unuep-\uE\|_{L^2(\Omega^\ep)}\leq B_{T} 
\frac{\sqrt\nu}{d_{\ep}^{(1+\mu)/2}},
\end{equation}
where $B_{T}$ is a constant depending only on $T$, $\|\omega_{0}\|_{L^1\cap
W^{1,\infty}(\R^2)}$,  and $\K$.
\end{theorem}

\begin{remark}
\begin{enumerate}
\renewcommand{\theenumi}{\roman{enumi}}
\renewcommand{\labelenumi}{\theenumi)}
\item In the case of a fixed number of obstacles shrinking homotetically to 
points (arbitrary located), the above theorem is a direct extension of the results in \cite{ILN09}, 
namely, under the condition $\ep\leq \tilde A \nu$ (with $\tilde A$ depending on the distance  and the number of  obstacles and on $\| \omega_{0}\|_{L^1\cap L^\infty(\R^2)}$) we obtain the convergence of the Navier-Stokes solution to the Euler 
solution with a rate of convergence of order $\sqrt\nu$.
\item The fact that our Euler correctors (see Proposition~\ref{prop:uepest}) 
are $\ep$-dependent, but not $\nu$-dependent, 
can be seen as a consequence of the fact that the local Reynolds number, built 
in terms of the characteristic size of the obstacles, stays of order one as 
$\nu$ and $\ep$ go to zero in the regime of Theorem~\ref{th:main1}. In this 
situation, the boundary layers are negligible, as already observed in 
\cite{ILN09}. In fact, the proof of Theorem~\ref{th:main1} is similar to the 
proof of Kato's criterion, which implies the vanishing viscosity limit 
\cite{Kato84,TW98,W01}.
\item The additional factor $\frac{1}{d_{\ep}^{(1+\mu)/2}}$ 
in \eqref{ineq.main} can be viewed as an effect of the homogenization of the
porous medium. Using standard correctors as for the Laplace problem (see
\cite{Allaire90a,Allaire90b,CM82,Tartar}), one would obtain a smaller bound 
of the form $\frac{1}{d_{\ep}^{(1+\mu)/2} \sqrt{|\ln\varepsilon|}}$. The limit
of $d_{\ep}^{(1+\mu)/2} \sqrt{|\ln\varepsilon|}$ determines precisely when the
so-called ``strange term'' in the homogenization appears. However, this
improvement is less relevant from an application standpoint, as discussed
next, and would require a significant amount of additional technical work
(cf. the discussion in Remark~\ref{rem:end}).
\item Refined
consequences of our main result arise by considering
specific regimes for the ratio of particle distance to particle size, as
discussed at
the end of Subsection \ref{sec:domain}. For instance, when $\mu =1$ (inclusions
are
uniformly distributed in a square) and $d_{\ep} = 1/\sqrt{|\ln \ep|}$, the
critical distance for the homogenization problem, if the viscosity $\nu$
satisfies \ $C \ep \sqrt{|\ln \ep|} \leq \nu \ll 1/|\ln \ep|$ --- for example,
$\nu = \ep^\alpha$ for some $\alpha\in (0,1)$ --- our result gives convergence
$\unuep \to \uE$, whereas if viscosity is kept fixed as $\ep \to 0$ $\unuep$
does not converge to the Navier-Stokes solutions in the full plane, as recalled
already. Similar conclusions can be drawn for smaller inter-particle distances,
namely the theorem implies the convergence to $u^E$ if $\ep\ll d_{\ep}^3$, as in the case $d_{\ep} =
\ep^\alpha$ with $\alpha< 1/3$. When $\mu =0$, the critical distance between
inclusions for the homogenization problem \cite{Allaire90b}
 is  $d_{\ep} = 1/|\ln \ep|$, whereas in our result the limit holds when $\ep\ll
d_{\ep}^{3/2}$, as for the case $d_{\ep} =
\ep^\alpha$ with $\alpha< 2/3$.
 \end{enumerate}
\end{remark}

\section{ Preliminaries} \label{sec:prelim}

We begin by presenting some basic elliptic estimates and estimates on the Euler 
solution that will be used throughout the article.

\subsection{Basic elliptic estimates}\label{sec:elliptic}

For the energy estimate in Section~\ref{sec:energy}, we will need to 
approximate a divergence-free vector field $u$ defined in $\R^2$ by divergence 
vector fields $u^\varepsilon$ verifying the Dirichlet boundary condition on 
$\partial \Omega^\varepsilon$.  To this end, we introduce an appropriate 
cut-off function $\phiep$ as follows.

As we consider the case where $d_{\varepsilon}\geq \varepsilon$, we deduce from 
\eqref{Kijeps}-\eqref{zij} that
\[
\K_{i,j}^\varepsilon \subset z_{i,j}^\varepsilon +\ep [-1,1]^2 \quad \forall 
(i,j)
\]
and that
\begin{equation}\label{U disjoint}
\Big(z_{i,j}^\varepsilon + \ep(-2,2)^2\Big)\bigcap \Big(z_{p,q}^\varepsilon + 
\ep(-2,2)^2\Big) = \emptyset \quad \forall (i,j) \neq (p,q).
\end{equation}

We let $\phi\geq 0$ be a smooth,  cut-off function such that:
\begin{equation}\label{eq:phi}
    \begin{cases}
         \phi(x) \equiv 1, &  |x|_{\infty} \leq 3/2, \\
         \phi(x) \equiv 0, & |x|_{\infty} \geq 2, \quad   
    \end{cases}
\end{equation}
where $|x|_\infty = \max_{i} |x_i|$,
and set 
\begin{equation}\label{phi ep}
\phi^{\varepsilon}(x) = 1- \sum_{i=1}^{n_{1}^\varepsilon} 
\sum_{j=1}^{n_{2}^\varepsilon}  \phi\left(\frac{x-z_{i,j}^\ep}\ep \right), 
\qquad x\in \RR^2.
\end{equation}
Then, $0\leq \phi^\ep \leq 1$ in $\R^2$,  $ \phi^\ep\equiv 1 $ away from the 
obstacles and $\phi^\ep\equiv 0$ on a small neighborhood of the obstacles. More 
precisely, denoting
\begin{equation}\label{def Aep}
      \Aep  :=\bigcup_{i,j} \Big(z_{i,j}^\varepsilon + \ep(-2,2)^2\setminus 
\ep \K \Big),
\end{equation}
we observe that  $\supp (1- \phi^\ep) \subset \Aep$. As we have 
the bound \eqref{n1n2} on the number of obstacles $n_{1}^\ep
n_{2}^\ep$,
we can easily estimate the Lebesgue measure 
of $\Aep$ and conclude that
\begin{equation}   \label{eq:phiepest} 
 \begin{aligned}
 \|1- \phi^\ep\|_{L^p} + \ep \|\nabla \phi^\ep\|_{L^p}  \leq C_{p} 
\frac{\ep^{2/p}}{d_{\varepsilon}^{(1+\mu)/p}}, 
  \qquad 1\leq p\leq \infty,  
 \end{aligned}
\end{equation}
with $C_{p}$  a constant independent of $\varepsilon$.

As $\phi^{\ep} u$ is not divergence free, the stream function $\psi$, 
satisfying $u=\nabla^\perp \psi$, can be truncated instead as in 
\cite{IK09,ILN03,ILN06,ILN09}, that is, one sets $u^\varepsilon:=\nabla^\perp 
(\phi^{\ep} \psi)$, which is divergence free and vanishes at the boundary.
In this case, one needs to compensate the growth of the gradient of 
$\phi^{\varepsilon}$ as the obstacles shrink, which is of order $1/\ep$.
For a single inclusion shrinking to a point, which we identify with the origin, 
one can choose $\psi$ such that $\psi(0)=0$ and then $u^\varepsilon = \phi^{\ep} 
u + \psi \nabla^\perp \phi^{\varepsilon}=\mathcal{O}(1)$. Of course, such a 
procedure cannot be applied if we consider more than one obstacle or if the 
obstacle is shrinking to a curve instead to a point.
Here we present an alternate way to truncate divergence-free vector fields 
so that they have support in the complement of several disjoint obstacles. The 
following method was used in \cite{Lac-NS3D} to treat the case of an obstacle 
shrinking to a curve and it is related to the Bogovski{\u\i} operator (see 
\cite{Galdi94}).

For $1\leq p<\infty$ and for any $\ep>0$ fixed, we define an equivalent norm 
$\| \cdot \|_{W^{1,p}_{\ep}}$ in the Sobolev space $W^{1,p}$ by
\begin{equation} \label{eq:W1pep}
\| f \|_{W^{1,p}_{\ep}} := \left(\frac1{\ep^p} \| f \|_{L^{p}(\Aep)}^p + \| 
\nabla f \|_{L^{p}(\Aep)}^p\right)^{1/p}.
\end{equation}

\begin{lemma}\label{lem:hep} Given $1<p <\infty$, there exists a 
constant  $\widetilde{C}_{p}$ depending only on $p$, 
such that the following holds: given any divergence-free vector field 
$u(t,\cdot)\in L^\infty(\R^2)$ for all $t\in [0, \infty)$ and 
given any $\ep>0$,
the problem
 \begin{equation*}
      \dive \hep = \nabla \phi^\ep\cdot u, \quad
       \text{on }  [0,\infty) \times \Aep, \\
\end{equation*} 
has a solution $h^\varepsilon(t,\cdot)\in 
W_{0}^{1,p}(\Aep)$ satisfying
\begin{gather*}
 \|\hep(t,\cdot) \|_{W^{1,p}_{\ep}}  \leq  \widetilde{C}_{p} \|\nabla \phiep\cdot u(t,\cdot)\|_{L^{p}(\Aep)} \quad \forall t\in [0,\infty),\\
  \|\hep(t,\cdot)-\hep(s,\cdot)  \|_{W^{1,p}_{\ep}} \leq  \widetilde{C}_{p} \|\nabla \phiep\cdot (u(t,\cdot)-u(s,\cdot) )\|_{L^{p}(\Aep)}\quad  \forall t,s\in [0,\infty).
\end{gather*}
Moreover, if $\frac{\partial u(t,\cdot)}{\partial t}\in L^{p}(\Aep)$ then
\[
\left\|\frac{\partial \hep(t,\cdot)}{\partial t} \right\|_{W^{1,p}_{\ep}}  \leq   \widetilde{C}_{p} \left\|\nabla \phiep\cdot \frac{\partial u(t,\cdot)}{\partial t} \right\|_{L^{p}(\Aep)}.
\]
\end{lemma}

Later on in Proposition \ref{prop:uepest}, we will employ this Lemma to
construct a suitable corrector to the Euler solution that is supported away
from the obstacles. Informally,
extending $h^\ep$ by zero, we readily verify that  $u^\ep := \phi^\ep\, u -\hep$
is 
divergence free, agrees with $u$ away from the obstacles, and the $L^{p}$ 
norm of $u^\ep$ is of order $\|u\|_{L^p}+\ep \| \nabla \phiep \|_{L^p} \|u\|_{L^\infty}$ (cutting the stream function $\psi$ would 
instead give a bound of order $\|u\|_{L^p}+ \| \nabla \phiep \|_{L^p} \|\psi\|_{L^\infty}$).

\begin{proof}[Proof of Lemma~\ref{lem:hep}] 
\ We start by introducing the 
bounded open set:
\begin{equation}\label{def U}
U:= (-2,2)^2\setminus \K.
\end{equation}
As $U$ satisfies the cone condition (see e.g. \cite[Rem. 
III.3.4]{Galdi94} for the definition), Theorem III.3.1 and 
Exercise III.3.6 (solvable thanks to Remark III.3.3) in \cite{Galdi94} state 
that there exists $\widetilde{C}_{p}$ with the following property. 
Given  any function $f(t,x)$ such that $f(t,\cdot)\in L^{p}(U)$, and  $\int_{U} 
f(t,\cdot)=0$ for all $t$, there exists a solution 
$h(t,\cdot)\in W_{0}^{1,p}(U)$ of the problem:
 \begin{equation}\label{eq div}
      \dive h = f, \quad
       \text{on }  [0,\infty) \times U, \\
\end{equation} 
such that 
\begin{gather}
 \|h(t,\cdot) \|_{W^{1,p}(U)}  \leq  \widetilde{C}_{p} \|f(t,\cdot)\|_{L^{p}(U)} \quad \forall t\in [0,\infty), \label{h1}\\
 \|h(t,\cdot)-h(s,\cdot)  \|_{W^{1,p}(U)} \leq  \widetilde{C}_{p} \|f(t,\cdot)-f(s,\cdot) \|_{L^{p}(U)}\   \forall t,s\in [0,\infty).\label{h3}
\end{gather}
Moreover, if $\frac{\partial f(t,\cdot)}{\partial t}\in L^{p}(U)$ then
\[
\left\|\frac{\partial h(t,\cdot)}{\partial t} \right\|_{W^{1,p}(U)}  \leq   \widetilde{C}_{p} \left\| \frac{\partial f(t,\cdot)}{\partial t} \right\|_{L^{p}(U)}.
\]
Solutions are non unique, and the 
main difficulty addressed in \cite{Galdi94} is how to obtain a solution $h$
vanishing at 
the boundary and verifying the estimates with constant $\widetilde{C}_{p}$ depending 
only on $U$ and $p$. The existence of such a solution follows from  an 
explicit representation formula due to Bogovski{\u\i} 
\cite{Bogovskii79,Bogovskii80}.

We next utilize this result to construct a solution $h_\ep$ on 
the domain
\[
\Aep=\bigcup_{i,j} \Big(z_{i,j}^\ep+\ep U\Big).
\]
Let $u$ be a given vector field on $\RR^2$ such that
$u(t,\cdot)\in L^\infty(\R^2)$  for any $t\in [0,\infty)$. For any $i=1,\dots, 
n_{1}^\ep$ and $ j=1,\dots,n_{2}^\ep$ fixed, the function:
\[
f_{i,j}(t,x):= \ep\left(\nabla \phi^\ep \cdot u \right) (t, z_{i,j}^\ep + \ep x)
\]
is defined on the set $U$, is bounded (hence it belongs to 
$L^{p}(U)$), and satisfies for each $t$:
\begin{align*}
 \int_{U}f_{i,j}(t,\cdot)&= \frac1{\ep}\int_{z_{i,j}^\ep+\ep U} \div (\phi^\ep  u)(t, y)\, d y =\frac1{\ep}\int_{z_{i,j}^\ep+\ep \partial U}\phi^\ep  u \cdot n\, d s\\
 &= \frac1{\ep}\int_{z_{i,j}^\ep+\ep \partial (-2,2)^2}  u \cdot n\, d s = \frac1{\ep}\int_{z_{i,j}^\ep+\ep (-2,2)^2} \div  u =0,
\end{align*}
where we have used twice that $u$ is divergence free on $\R^2$, that
$\phi^\ep\equiv 1$ on  $z_{i,j}^\ep+ \ep\partial (-2,2)^2$, and
that $\phi^\ep\equiv 0$ on  $z_{i,j}^\ep+ \ep\partial K$.

Therefore, we can apply Galdi's results recalled above with $f$ replaced by 
$f_{i,j}(t,x)$ to conclude that there exists  a function 
$h_{i,j}$ on $[0,\infty)\times U$ such that $h_{i,j}(t,\cdot)\in W_{0}^{1,p}(U)$ is a solution of \eqref{eq div} with $f=f_{i,j}$ and 
satisfies \eqref{h1}-\eqref{h3}. Finally, we extend $h_{i,j}$ by 
zero on $\R^2$ and we define:
\[
\hep(t,x):= \sum_{i,j} h_{i,j} \left(t, \frac{x-z_{i,j}^\ep}\ep\right).
\]
To finish the proof, we show that $\hep$ has all the properties stated in Lemma 
\ref{lem:hep}. By definition, 
$h^\ep(t,\cdot)\in  W_{0}^{1,p}(\Aep)$ for all $t$ and 
we have 
\begin{align*}
\div \hep(t,x) &= \frac1\ep \sum_{i,j} (\div h_{i,j}) \left(t, \frac{x-z_{i,j}^\ep}\ep\right) = 
 \frac1\ep \sum_{i,j} f_{i,j} \left(t, \frac{x-z_{i,j}^\ep}\ep\right) \text{\Large  $\mathds{1}$}_{z_{i,j}^\ep + \ep U}\\
 &= \nabla \phi^\ep \cdot u (t,x),
\end{align*}
using that the obstacles are separated and $\supp \nabla \phi^\ep \subset \bigcup \{z_{i,j}^\ep + \ep U\}$. In the 
formula above, {\Large  $\mathds{1}$}$_{z_{i,j}^\ep + \ep U}$ denotes the 
characteristic function  of the set $z_{i,j}^\ep + \ep U$.
Next, for any $t$, we compute
\begin{align*}
\| \hep(t,\cdot) \|_{W^{1,p}_{\ep}}^{p} =&
\frac1{\ep^{p}} \sum_{i,j}\int_{z_{i,j}^\ep + \ep U} |h_{i,j}|^{p} \left(t, \frac{x-z_{i,j}^\ep}\ep \right) \, dx\\
&+\sum_{i,j}\int_{z_{i,j}^\ep + \ep U} \left| \frac1{\ep} \nabla h_{i,j}\right|^{p} \left(t, \frac{x-z_{i,j}^\ep}\ep \right) \, dx\\
=&\ep^{2-p} \sum_{i,j} \| h_{i,j}(t,\cdot) \|_{W^{1,p}(U)}^{p}\leq \ep^{2-p} \widetilde{C}_{p}^{p} \sum_{i,j} \| f_{i,j}(t,\cdot) \|_{L^{p}(U)}^{p}\\
\leq&\ep^2 \widetilde{C}_{p}^{p}  \sum_{i,j} \int_{ U}  | \nabla \phi^\ep \cdot u 
|^{p_1}
(t,z_{i,j}^\ep+\ep x) \, dx \\
\leq& \widetilde{C}_{p}^{p}  \| \nabla \phi^\ep \cdot u(t,\cdot) \|_{L^{p}(\Aep)}^{p}.
\end{align*}
A similar calculation applies to derive the bounds on  $\| \hep(t,\cdot) - 
\hep(s,\cdot)\|_{W^{1,p}_{\ep}}$, and $\| \partial_{t}\hep(t,\cdot) 
\|_{W^{1,p}_{\ep}}$.
\end{proof}

\begin{remark}
In the proof of Lemma~\ref{lem:hep} above, we have used several times 
the fact that the obstacles are well separated (see Equation \eqref{U 
disjoint}). Therefore,  our analysis can be applied only to the case 
$d^\ep \geq C \ep$ for some $C$ independent of $\ep$. For distances $d^\ep \ll  
\ep$, we would need to understand the behavior of $\widetilde{C}_{p}$ on 
domains of the form  $ (-1,1)^2\setminus \rho \K$  
as $\rho\to 1^-$.
\end{remark}

In the same way, changing variables in the classical Poincar\'e inequality, we 
obtain the following result. The proof can be found in \cite[Lem. 3]{ILN09}. We 
sketch it here for the reader's convenience.

\begin{lemma}\label{lem:poincare} 
Let $\chi$ be a smooth cutoff function such that $\chi\equiv 1$ in $B(0,3)$ and 
$\chi\equiv 0$ in $B(0,4)^c$.
There exists a constant $K_{1}$ depending  only on $\K$  such that for any $u$ verifying
\[
\chi u \in H^1_{0}(\Omega^\ep),
\]
then
\begin{equation}\label{eq:Poincare}
    \| u \|_{L^2(\Aep)} \leq \ep K_1 \|\nabla u\|_{L^2(\Aep)},
\end{equation}
where $\Aep$ is given in Equation \eqref{def Aep}.
\end{lemma}

\begin{proof}
We note that $\Aep \subset B(0,1)$ for all $\ep\leq 1$.

By the usual Poincar\'e inequality,
\[
\| u \|_{L^2(U)} \leq K_1 \|\nabla u\|_{L^2(U)},
\]
for all $u$ such that $\chi u \in H^1_{0}(
B(0,4)
\setminus \K)$, where $U$ is defined in \eqref{def U}.
Then, from the definition of $\Aep$ given in \eqref{def Aep}, it easily follow 
that
\begin{align*}
 \| u \|_{L^2(\Aep)}^2 &= \sum_{i,j} \int_{z_{i,j}^\ep+\ep U} |u(x)|^2\, dx
= \sum_{i,j} \int_{U} \left| u\left( z_{i,j}^\ep + \ep y \right)\right|^2\, \ep^2 dy\\
&\leq \sum_{i,j} \ep^2 K_{1}^2 \int_{U} \left|\ep \nabla u\left( z_{i,j}^\ep + \ep y \right)\right|^2\,  dy
= \ep^2 K_{1}^2 \| \nabla u \|_{L^2(\Aep)}^2,
\end{align*}
which ends the proof.
\end{proof}

We finish this subsection with a {\em Sobolev}-type  embedding.

\begin{lemma}\label{lem:ladyz}
Let $\chi$ be a cutoff function as in Lemma~\ref{lem:poincare}, and let 
$\Aep$ again be given in \eqref{def Aep}.
Then, there exists a constant $K_{2}$ depending  only on $\K$ such that for any 
$u$ verifying \
$\chi u \in H^1_{0}(\Omega^\ep)$, 
it holds
\begin{equation*}
    \| u \|_{L^4(\Aep)} \leq \sqrt{\ep} K_2 \|\nabla u\|_{L^2(\Aep)}.
\end{equation*}
\end{lemma}

\begin{proof}
Bringing together the Sobolev embedding of $H^{1}(U)$ in $L^4(U)$ and the
Poincar\'e inequality, we have
\[
\| u \|_{L^4(U)} \leq K \| u\|_{H^1(U)} \leq K K_2 \|\nabla u\|_{L^2(U)},
\]
for all $u$ such that $\chi u \in H^1_{0}(\R^2 \setminus \K)$.

Hence, a change of variables gives
\begin{align*}
 \| u \|_{L^4(\Aep)}^2 &= \Big(\sum_{i,j} \int_{z_{i,j}^\ep+\ep U} |u(x)|^4\, dx\Big)^{1/2}
 \leq \sum_{i,j} \Big(\int_{z_{i,j}^\ep+\ep U} |u(x)|^4\, dx\Big)^{1/2}\\
&\leq \sum_{i,j}  \Big( \int_{U} \left| u\left( z_{i,j}^\ep + \ep y \right)\right|^4\, \ep^2 dy\Big)^{1/2}\\
&\leq \sum_{i,j} K K_{1}\ep \int_{U} \left|\ep \nabla u\left( z_{i,j}^\ep + \ep y \right)\right|^2\,  dy
= \ep K_{2}^2 \| \nabla u \|_{L^2(\Aep)}^2.
\end{align*}
\end{proof}

\subsection{Basic estimates on the Euler equations}\label{sec:Euler}

We end these preliminaries by collecting some results  on the Euler
solution $u^E$ which will be used mainly in Section~\ref{sec:energy}. Except 
for Proposition~\ref{prop:uepest}, the statements presented here are well 
known and complete proofs can be found, for example, in
\cite{MajdaBertozzi}.

As mentioned in the Introduction, we give the initial velocity
 as $u_{0}=K_{\R^2}[\omega_{0}]$, where $K_{\RR^2}$ is the
Biot-Savart kernel defined in\eqref{u0} and the initial vorticity $\omega_{0}\in
C^\infty_{c}(\R^2)$. Hence, by a result of McGrath \cite{McGrath} there is
a unique global strong solution $\uE$ of the Euler equations \eqref{eq:Euler} in
the full plane. 

The vorticity $\omega:= \curl\uE$ verifies in a weak sense the transport equation:
\[
\partial_{t}\omega + \uE \cdot \nabla \omega = 0,
\]
which allows us to deduce the conservation of the $L^p$ norm of the vorticity:
\[
\| \omega(t,\cdot) \|_{L^p(\R^2)} = \| \omega_{0} \|_{L^p(\R^2)}, \quad \forall t>0,\quad 1\leq p \leq \infty.
\]
By a classical estimate on the Biot-Savart kernel, we obtain that $\uE$ is 
uniformly bounded:
\begin{equation}\label{uE infty}
\| \uE \|_{L^\infty(\R^+\times \R^2)} \leq C \| \omega \|_{L^\infty(\R^+\times \R^2)}^{1/2} \|  \omega \|_{L^\infty(\R^+;L^1( \R^2))}^{1/2}\leq C\| \omega_{0}\|_{L^1\cap L^\infty}
\end{equation}
and by the Calder\'on-Zygmund inequality, we infer that for all $p\in (1, \infty)$:
\begin{equation}\label{uE CZ}
\| \nabla \uE \|_{L^\infty(\R^+;L^p (\R^2))} \leq C_{p} \|  \omega \|_{L^\infty(\R^+;L^p( \R^2))}\leq C_{p}\| \omega_{0}\|_{L^1\cap L^\infty}.
\end{equation}
Hence,  the nonlinear term $\uE\cdot \nabla \uE$ belongs  to 
$L^\infty(\RR^+;L^2\cap L^4(\R^2))$. Since  $p^E$ is a solution of 
$\Delta p^E=\div (\uE\cdot \nabla \uE)$, up to choosing $p^E$ with zero mean value on $B(0,2)$, it belongs to $L^\infty(\RR^+; H^1\cap 
W^{1,4}(B(0,2)))$
and we deduce from elliptic estimates and the Sobolev embedding that
\begin{equation}\label{pE}
p^E \in  L^\infty(\R^+\times B(0,2)).
\end{equation}
These estimates also imply that $u_{t}^E$ is in $L^\infty(\R^+;L^4(B(0,2))$.

Since the Calder\'on-Zygmund inequality is not true for $p=\infty$,
there is no bound on
$\| \nabla \uE(t,\cdot )\|_{L^\infty(\R^2)}$ uniformly in time. However, 
the following well-known estimate  \cite{yudo63} holds:
\begin{equation}\label{grad uE}
\| \nabla \uE(t,\cdot )\|_{L^\infty(\R^2)} \leq C_{0} e^{C_{0}t} \quad \forall t\geq 0,
\end{equation}
where $C_{0}$ depends on 
$\|\omega_{0}\|_{W^{1,\infty}}$.
(See  \cite{KiselevSverak} for a discussion of
results concerning the sharpness of the double exponential growth of $\|\nabla
\omega\|_{L^\infty}$.)
In this paper, we do not investigate the minimal regularity of the initial
data needed for our result to hold, as the vanishing viscosity limit is a
singular problem even for smooth data. In particular, it is known that
\eqref{grad uE} holds under much weaker conditions (e.g. $\omega_0$ in the Besov
space $B^{2/p}_{p,1}$ \cite{Vi98}).

In the following proposition, we utilize Lemma~\ref{lem:hep} to construct a 
corrector to the truncated velocity $\phi^\ep \uE$  to meet 
the no-slip boundary conditions and the divergence-free condition, which will 
be employed for the energy estimate in Section~\ref{sec:energy}.

\begin{proposition} \label{prop:uepest} Let $\omega_{0}\in C^\infty_{c}(\R^2)$
and let $\uE$ be the solution of the Euler equations 
\eqref{eq:Euler} with initial condition $u_{0}$, given in terms of $\omega_{0}$ 
in \eqref{u0}). For any $\varepsilon>0$, let $\phi^\ep$ and $\Aep$ be defined 
as in \eqref{phi ep}-\eqref{def Aep}. Then,
there exists $\hep \in 
L^\infty(\R^+,W^{1,4}_{0}(\Aep))$ such that, after extending $\hep$ in $\R^2$ by
zero, and setting 
\[
u^\varepsilon:= \phi^\ep \uE - \hep,
\]
we have that $\uep$ is divergence free and identically zero on $\partial
\Omega^\varepsilon$. 
Moreover,  there exists a constant $K_E$,
dependent on $u^E$ but independent of
$\varepsilon$,
such that 
for all $\varepsilon>0$:
\begin{enumerate}
\item $\| \hep \|_{L^\infty(\R^+;L^4(\R^2))} \leq K_E \frac{\sqrt{\ep}}{   d_{\varepsilon}^{(1+\mu)/4} }$; \label{i:uepest.a}
\item $\| \partial_{t} \hep \|_{L^\infty(\R^+;L^2(\R^2))} \leq K_E \frac{\sqrt{\ep}}{   d_{\varepsilon}^{(1+\mu)/4} } $; \label{i:uepest.b}
\item $\| \nabla \hep   \|_{L^\infty(\R^+;L^2(\R^2))} \leq K_E \frac{1}{d_{\varepsilon}^{(1+\mu)/2}}  $;\label{i:uepest.c}
\item $\| \uE - u^\ep   \|_{L^\infty(\R^+;L^4(\R^2))} \leq K_E  \frac{\sqrt{\ep}}{   d_{\varepsilon}^{(1+\mu)/4} }$.\label{i:uepest.d}
\end{enumerate}
\end{proposition}

\begin{proof}
The main idea is to use Lemma~\ref{lem:hep} with $p=4$. As 
$\uE$ is divergence free and uniformly bounded,
there exists $\hep \in
 W^{1,4}_{0}(\Aep)$ such that  $ \div \hep = \nabla \phi^\ep \cdot 
\uE$, so that
\[
\div u^\ep =  \nabla \phi^\ep \cdot \uE + \phi^\ep \div \uE - \div \hep \equiv 0.
\]
Using  \eqref{eq:phiepest} in conjunction with \eqref{uE infty}, the estimates 
in Lemma~\ref{lem:hep} give:
\[
\| \hep \|_{L^\infty(L^4)} \leq \widetilde{C}_{4} \ep \| \nabla \phi^\ep \cdot \uE 
\|_{L^\infty(L^4)} \leq \widetilde{C}_{4} C_{4}    \frac{\ep^{1/2}}{d_{\varepsilon}^{(1+\mu)/4}}  
   \| \uE \|_{L^\infty},
\]
which proves \eqref{i:uepest.a}. 
Since $u_{t}^E$ belongs to 
$L^\infty(L^4(B(0,2)))$, we also obtain \eqref{i:uepest.b} by  H\"older's
inequality:
\begin{align*}
 \| \partial_{t} \hep \|_{L^\infty(L^2)} 
&\leq  C_{4}    \frac{\ep^{1/2}}{d_{\varepsilon}^{(1+\mu)/4}}  \| \partial_{t} \hep \|_{L^\infty(L^4)} \\
 &\leq C_{4}     \frac{\ep^{1/2}}{d_{\varepsilon}^{(1+\mu)/4}} \widetilde{C}_{4} \varepsilon \| \nabla \phi^\ep\|_{L^\infty}     \| \uE_{t} \|_{L^\infty(L^4(\supp \nabla \phiep))}.
\end{align*}

We prove \eqref{i:uepest.c} in a similar way:
\[
\| \nabla \hep   \|_{L^\infty(L^2)} \leq  C_{4}    \frac{\ep^{1/2}}{d_{\varepsilon}^{(1+\mu)/4}}  \widetilde{C}_{4}  \| \nabla \phi^\ep \cdot \uE\|_{L^\infty(L^4)} \\
 \leq \frac{C}{d_{\varepsilon}^{(1+\mu)/2}}  \| \uE \|_{L^\infty}.
\]

Finally, to establish \eqref{i:uepest.d},  we observe that
\[
\| \uE-  \uep\|_{L^4} \leq  \|(1-\phiep) \uE\|_{L^4} +  \|\hep\|_{L^4} \leq C 
\frac{\ep^{1/2}}{d_{\varepsilon}^{(1+\mu)/4}},
\]
which concludes the proof.
\end{proof}

\begin{remark}
Since $\uE_{t}$ is bounded in $L^\infty(L^p)$ for any $p<\infty$, 
Lemma \ref{lem:hep} with $p=2$  also gives that
\[
\| \partial_{t} \hep \|_{L^\infty(L^2)} \leq \widetilde{C}_{2} \ep \| \nabla \phi^\ep \cdot 
\uE_{t} \|_{L^\infty(L^2)} \leq C     
\frac{\ep^{2/q}}{d_{\varepsilon}^{(1+\mu)/q}}     \| \uE_{t} \|_{L^\infty(L^p)},
\]
where $q=\frac{2p}{p-2}>2$ can be chosen as close as we want of $2$.

Nevertheless, for the sequel, we only need the case $q=4$ (the case
corresponding to the bound \eqref{i:uepest.b} in Proposition \ref{prop:uepest}
above).
\end{remark}

\begin{remark}\label{rem.KE}
As we did not use \eqref{grad uE} in the previous proof,
we can give a more precise dependence of the constant $K_{E}$  on the geometry
and the Euler solutions (which are both fixed throughout). In fact,  $K_E$
depends  only on $\K$ and $\| \omega_{0}\|_{L^1\cap L^\infty}$. In particular,
the bounds \eqref{i:uepest.a}, \eqref{i:uepest.c}, \eqref{i:uepest.d} in
Proposition  \ref{prop:uepest} are linear in $\| \uE \|_{L^\infty}$. Hence,
in these estimates $K_{E}$ is of the form $K\| \omega_{0}\|_{L^1\cap L^\infty}$,
where $K$ depends only on $\K$.
\end{remark}

\section{Convergence of the initial velocity} \label{sec:initial}

In this section, we discuss the convergence of the Navier-Stokes initial 
 data (taken dependent on $\ep$ only) to the Euler initial data in the energy 
norm, as $\ep$ goes to zero. 

Let $\omega_{0}\in C^\infty_{c}(\R^2 \setminus ([0,1]\times \{ 0\}))$ if $\mu
\in [0,1)$ and in $C^\infty_{c}(\R^2 \setminus ([0,1]^2))$ if $\mu =1$, then for
$\varepsilon$ small enough, $\omega_{0}\equiv 0$ on $\K_{i,j}^\varepsilon$ for
all $i,j$. For any $\ep
>0$ there exists a unique vector field $u_{0}^\ep$, which is solution of 
\eqref{u0eps} (see e.g. \cite{Kikuchi} for a proof). We choose  $R$ so that 
$\supp \omega_{0} \cup \partial \Omega^\ep \subset B(0,R)$. 
The function 
\[
\psi:\ z=x_{1}+ix_{2}\mapsto (u_{0}^\ep)_{1}(x_{1},x_{2})-i(u_{0}^\ep)_{2}(x_{1},x_{2})
\]
 verifies the 
Cauchy-Riemann equations on $B(0,R)^c$. Hence, $\psi$ is holomorphic and admits 
a Laurent series decomposition $\psi(z)=\sum_{k=1}^\infty \frac{c_{k}}{z^k}$, 
which allows us to conclude that $u_{0}^\ep \in L^\infty\cap 
L^{2,\infty}(B(0,R)^c)$ (where $L^{2,\infty}$ denotes the Marcinkiewicz weak 
$L^2$-space). By elliptic regularity,  $u_{0}^\ep$ is bounded in 
$B(0,R)\cap \Omega^\ep$, implying that
\begin{equation}\label{u0ep regularity}
u_{0}^\ep \in L^\infty\cap L^{2,\infty}(\Omega^\ep).
\end{equation}
This regularity of the initial data is sufficient  to apply the result of 
Kozono and Yamazaki \cite{KY95}, yielding existence and uniqueness of global 
solution to the Navier-Stokes equations \eqref{eq:NSE}.

The goal of this section is to prove that $u_{0}^\ep$ converges in $L^2$ to 
$u_{0}$ defined in \eqref{u0}. More precisely, we prove the following:

\begin{proposition}\label{prop:u0eps}
Let $\OMep$ be defined in \eqref{Kijeps}-\eqref{OM} with $d_{\ep}\geq \ep$ and 
let $\omega_{0}\in C^\infty_{c}(\R^2)$. There exists a constant $C$, 
which depends only on $\K$, such that, if $\omega_{0}$ is supported in $\Omega^\varepsilon$ then
\[
\| u_{0}^\varepsilon - u_{0}\|_{L^2(\Omega^\varepsilon)} \leq C \| \omega_{0}\|_{L^1\cap L^\infty} \frac{\ep |\ln \ep | }{d_{\ep}^{(1+\mu)/2}},
\]
where $u_{0}^\ep$ is the unique solution of \eqref{u0eps} and $u_{0}$ the unique solution of \eqref{eq:u0}.
\end{proposition}

This proposition follows from the analysis developed in \cite{BLM}. There, the 
authors have looked for the best condition on $\ep$ and $d_{\ep}$ ensuring 
convergence of $u_{0}^\varepsilon$ to $u_{0}$ in $L^2(\Omega^\varepsilon)$, but 
unfortunately, they did not explicitly state the rate of convergence. To obtain 
the rate, we now briefly review the results in \cite{BLM}. For brevity, we will
denote \ $M_{0}:=\| \omega_{0}\|_{L^1\cap L^\infty}$.

\subsection{Correction and decomposition} 
\def\Tc{\mathcal{T}}
\def\Tca#1{{\Tc^{\ep}_{#1}}}

We recall that $u_{0}$  has an explicit formula in terms of 
$\omega_{0}$ via the Biot-Savart kernel (Equation \eqref{u0}). However, there 
is no such formula available for  $u_{0}^\ep$. In addition, $u_{0}$ is 
not tangent to $\partial \OMep$ and hence, it is necessary to introduce an 
explicit corrector, which is given in terms of:
\begin{enumerate}
 \item  a cutoff function $\phi_{i,j}^\ep$ equal to $1$ close to $\K_{i,j}^\ep$:
\[
\phi_{i,j}^\ep(x):=   \phi\left(\frac{x-z_{i,j}^\ep}\ep \right) \quad \text{(
with $\phi$ given in \eqref{eq:phi})} ;
\]
\item a biholomorphism $\Tc: \ \K^c \to \R^2\setminus \overline{B(0,1)}$ such that $\Tc(z)=\beta z +h(z)$ for some $\beta \in \R^+$ and $h$ a bounded holomorphic function.
\end{enumerate}
Then, the correction is defined by
\[
v^\ep:= \nabla^\perp \psi^\ep,
\]
where 
\begin{align*}
\psi^\ep(x) :=& \frac1{2\pi} \int_{\OMep} \ln|x-y| \omega_{0}(y)\, d y\\
&- \frac1{2\pi}\sum_{i,j} \phi_{i,j}^\ep(x) \int_{\OMep} {\ln}\frac{\beta |x-y|}{\ep |\Tca{i,j}(x)-\Tca{i,j}(y)| }      \omega_{0}(y)\, d y\\
&+\frac{1}{2\pi}\sum_{i,j} \phi_{i,j}^\ep (x)\int_{\OMep}{\ln}\frac{ 
|\Tca{i,j}(x)|}{|\Tca{i,j}(x)-\Tca{i,j}(y)^*|} \omega_{0}(y) \, d y, 
\end{align*}
with
\[
\Tca{i,j}(x) := \Tc\left(\frac{x-z_{i,j}^\ep}\ep \right).
\]
Above, we have denoted by $y^*=\frac{y}{|y|^2}$ the conjugate point to $y$
across 
the unit circle in $\RR^2$. This formula is related to the 
Biot-Savart law outside one obstacle, and we can check that $v^\ep$ verifies 
the following properties:
\begin{equation*}\begin{split}
\div v^\ep =0 \text{ in } \Omega^\varepsilon, \quad  v^\ep \cdot n =0 \text{ on } \partial \Omega^\varepsilon, \\
  \int_{\partial \K_{i,j}^\varepsilon}v^\ep \cdot \tau\, ds =0 \text{ for all } i,j,\quad \lim_{x\to \infty} |v^\ep(x)|=0.
\end{split}\end{equation*}

Then, we decompose $u_{0}-v^\ep$ as 
\begin{equation} \label{decompo we}
u_{0}-v^\ep =\sum_{k=1}^4 w_{k}^\ep,
\end{equation}
where
\begin{equation*}\begin{split}
w_{1}^\ep(x)=&\frac1{2\pi} \sum_{i,j} \nabla^\perp \phi_{i,j}^\ep(x) \int_{\OMep} \ln \frac{\beta|x-y|}{\ep|\Tca{i,j}(x)-\Tca{i,j}(y)|}\omega_{0}(y)\, dy, \\
w_{2}^\ep(x)=&\frac1{2\pi} \sum_{i,j}\nabla^\perp \phi_{i,j}^\ep(x) \int_{\OMep} \ln \frac{|\Tca{i,j}(x)-\Tca{i,j}(y)^*|}{|\Tca{i,j}(x)|}\omega_{0}(y)\, dy, \\
w_{3}^\ep(x)=&\frac1{2\pi}\sum_{i,j}\phi_{i,j}^\ep(x) \int_{\OMep} \Biggl(\frac{(x-y)^\perp}{|x-y|^2}- (D\Tca{i,j})^T(x)\frac{(\Tca{i,j}(x)-\Tca{i,j}(y))^\perp}{|\Tca{i,j}(x)-\Tca{i,j}(y)|^2} \Biggl) \omega_{0}(y)\, dy,    \\
w_{4}^\ep(x)=&\frac1{2\pi} \sum_{i,j}\phi_{i,j}^\ep(x) (D\Tca{i,j})^T(x) \int_{\OMep} \Biggl(\frac{\Tca{i,j}(x)-\Tca{i,j}(y)^*}{|\Tca{i,j}(x)-\Tca{i,j}(y)^*|^2}- \frac{\Tca{i,j}(x)}{|\Tca{i,j}(x)|^2}\Biggl)^\perp \omega_{0}(y)\, dy,
\end{split}\end{equation*}
where we have used that $\omega_{0}$ is supported in $\Omega^\varepsilon$.
In \cite{BLM} and in \cite{LLL}, the explicit formula for $w_{k}^\ep$ is 
utilized to prove that $v^\ep$ converges to $u_{0}$ as $\ep\to 0$. 
In \cite{BLM}, a convergence analysis  was also performed for $d_{\ep} \ll 
\ep$, which complicates the choice of the cutoff function. In our 
case  where $d_{\ep}\geq \ep$, the cutoff $\phi_{i,j}^\ep$ defined above is 
sufficient to apply the estimates in \cite{BLM}. We recall these estimates, 
adapted to the set-up of this paper, in the remainder of the section.

\subsection{Estimates on \protect$w_{1}^\ep$ and \protect$w_{3}^\ep$}

Following the proofs of Proposition 3.3 and Proposition 3.4 in \cite{BLM}, we 
first tackle estimates for $w_{1}^\ep$ and $w_{3}^\ep$. When 
$\K=\overline{B(0,1)}$, $\Tc={\rm Id}$ (so $\beta=1$) and 
$w_{1}^\ep=w_{3}^\ep \equiv 0$. In the general case, the main idea is to use 
that
\[
\Tca{i,j}(x)= \Tc \Big(\frac{x-z_{i,j}^\ep}{\ep} \Big)\sim \beta \frac{x-z_{i,j}^\ep}{\ep} = \beta \ {\rm Id}\Big(\frac{x-z_{i,j}^\ep}{\ep}\Big).
\]

In \cite[Proposition 3.3]{BLM}, it was proved that
\[
\Bigl\| \int_{\OMep} \ln \frac{\beta|x-y|}{\ep|\Tca{i,j}(x)-\Tca{i,j}(y)|}\omega_{0}(y)\, d y  \Bigl\|_{L^\infty(\OMep)} \leq CM_{0}\ep |\ln \ep|,
\]
where $C$ depends only on  $\K$. Using that the 
supports of $\nabla \phi_{i,j}^\ep$ are disjoint for distinct $i,j$ 
(see \eqref{U disjoint}), we deduce that:
\[
\| w_{1}^\ep \|_{L^2(\OMep)} \leq CM_{0}\ep |\ln \ep| \frac1\ep \| \nabla\phi \|_{L^\infty}\frac{\ep}{d_{\ep}^{(1+\mu)/2}} =  CM_{0} \frac{\ep |\ln \ep| }{d_{\ep}^{(1+\mu)/2}}.
\]

In \cite[Proposition 3.4]{BLM}, it was also shown that, for all $i,j$ and $x\in 
\Omega^\ep$, it holds
\begin{multline*}
 \Bigg|\int_{\OMep} \Biggl(\frac{(x-y)^\perp}{|x-y|^2}- (D\Tca{i,j})^T(x)\frac{(\Tca{i,j}(x)-\Tca{i,j}(y))^\perp}{|\Tca{i,j}(x)-\Tca{i,j}(y)|^2} \Biggl) \omega_{0}(y)\, dy \Bigg|\\
   \leq CM_{0}\Big( \ep^{1/2}+\frac{\ep}{|x-z_{i,j}^\ep|}\Big),
\end{multline*}
where $C$ depends only on $\K$. As there exists $\delta>0$ such that 
$B(0,\delta)\subset \K$, $\supp \phi_{i,j}^\ep \subset B(z_{i,j}^\ep, 
2\sqrt{2}\ep)\setminus B(z_{i,j}^\ep,\delta\ep)$, and we can estimate the $L^2$ 
norm of $w_{3}^\ep$ as follows:
\begin{align*}
\| w_{3}^\ep \|_{L^2(\OMep)} & \leq CM_{0}\ep^{1/2} \frac{\ep}{d_{\ep}^{(1+\mu)/2}} + CM_{0}\ep \Big( \frac{1}{d_{\ep}^{1+\mu}} \int_{B(0, 2\sqrt{2}\ep)\setminus B(0,\delta\ep)} \frac{d x}{|x|^2}\Big)^{1/2}\\
&\leq CM_{0}\frac{\ep |\ln \ep |^{1/2}}{d_{\ep}^{(1+\mu)/2}}.
\end{align*}

\subsection{Estimates on $w_{2}^\ep$ and $w_{4}^\ep$}

When $\K=\overline{B(0,1)}$, we have $\Tca{i,j}(y)^* = \ep^2 
\dfrac{y-z_{i,j}^\ep}{|y-z_{i,j}^\ep|^2}$. 
We observe that, even though  $|y-z_{i,j}^\ep|$ can be of order  $\ep$ (this 
is, in fact, the case if $y\in \partial B(z_{i,j}^\ep,\ep)$), $\Tca{i,j}(y)^*$ 
is still small compare to $\Tca{i,j}(x)$, since $|\Tca{i,j}(x)|\geq 
1$. Hence,  $w_{2}^\ep$ and $w_{4}^\ep$ should be small.

More precisely, in \cite[Proposition 3.5]{BLM}, it was proved that, for all 
$i,j$,
\[
\Bigg\|  \int_{\OMep} \ln \frac{|\Tca{i,j}(x)-\Tca{i,j}(y)^*|}{|\Tca{i,j}(x)|}\omega_{0}(y)\, dy \Bigg\|_{L^\infty(\supp \nabla \phi_{i,j}^\ep)} \leq CM_{0} \ep
\]
where $C$ depends only on  $\K$. Therefore, we find that
\[
\| w_{2}^\ep \|_{L^2(\OMep)} \leq C M_{0} \ep  \frac1\ep \| \nabla\phi \|_{L^\infty}\frac{\ep}{d_{\ep}^{(1+\mu)/2}} =  C M_{0} \frac{\ep }{d_{\ep}^{(1+\mu)/2}}.
\]

Concerning the estimate on $w_{4}^\ep$, the proof of Proposition 3.6 in 
\cite{BLM} shows that, for all $i,j$ and $x\in \Omega^\ep$,
\begin{multline*}
   \Bigg|(D\Tca{i,j})^T(x) \int_{\OMep} \Biggl(\frac{\Tca{i,j}(x)-\Tca{i,j}(y)^*}{|\Tca{i,j}(x)-\Tca{i,j}(y)^*|^2}- \frac{\Tca{i,j}(x)}{|\Tca{i,j}(x)|^2}\Biggl)^\perp \omega_{0}(y)\, dy \Bigg| \\ \leq CM_{0}\Big( \ep+\frac{\ep}{\ep |\Tca{i,j}(x)|}\Big),
\end{multline*}
where $C$ depends only on  $\K$. Therefore, after the change of 
variable $z=\ep \Tca{i,j}(x)$, we can estimate the $L^2$ norm in the same way 
as for $w_{3}^\ep$ (see \cite{BLM} for details), and we conclude that
\[
\| w_{4}^\ep \|_{L^2(\OMep)} \leq   CM_{0} \frac{\ep |\ln \ep |^{1/2} }{d_{\ep}^{(1+\mu)/2}}.
\]

\subsection{End of the proof of Proposition \ref{prop:u0eps}}

Collecting the estimates on  $w_{k}^\ep$ for $k=1,\ldots,4$ gives that 
for any $\ep>0$, there exists a constant $C>0$ such that 
\[
 \| u_{0}-v^\ep \|_{L^2(\OMep)} \leq CM_{0} \frac{\ep |\ln \ep | 
}{d_{\ep}^{(1+\mu)/2}}.
\]
Finally, we observe that $u_{0}^\ep -v^\ep$ is the Leray projection of 
$u_{0}-v^\ep$ (see \cite[Sect. 3.2]{BLM}). Hence, the $L^2$-orthogonality of 
this projection implies
\[
\| u_{0}^\ep-v^\ep \|_{L^2(\OMep)} \leq \| u_{0}-v^\ep \|_{L^2(\OMep)} \leq C M_{0}\frac{\ep |\ln \ep | }{d_{\ep}^{(1+\mu)/2}}.
\]
We conclude the proof of Proposition~\ref{prop:u0eps} by the triangle 
inequality.

\section{Energy estimate} \label{sec:energy}

In this section, we establish convergence of
the solution $\unuep$ of the
Navier-Stokes equations in $\Omega^\ep$ to the solution $\uE$ of the Euler
equations in the whole plane, provided $\frac{\ep}{d_\ep^{1+\mu/2}} \leq A\,
\nu$ for some given positive constant $A$. The proof is
performed by an energy estimate on the difference  $\unuep-u^\ep$, where
$u^\ep$ vanishes on $\partial \Omega^\ep$ and is close to $\uE$ (see Proposition 
\ref{prop:uepest}). This strategy is related to the proof of Kato's criterion 
(see \cite{Kato84,TW98,W01}) and was used in \cite{ILN09} to treat the case of 
one shrinking obstacle.

Heuristically, it is possible to pass to the zero-viscosity limit even in the 
presence of the boundary layer due to the no-slip boundary condition, since 
under the conditions of the theorem below the contribution to the energy from 
the Euler solution restricted to the obstacles is negligible. Nevertheless, 
in the case of infinitely-many obstacles a dependence on $d_\ep$, which reduces 
the expected rate of convergence of $\sqrt{\nu}$, valid for one shrinking 
obstacle, appears in the convergence estimate \eqref{eq:L2conv}. This 
dependence can be interpreted as a ``ghost'' of the perforated domain.

\begin{theorem} \label{th:main}
Let $\uE$ be the solution of the Euler equations \eqref{eq:Euler} on 
$[0,\infty)\times \RR^2$ with initial condition $u_{0}$ (given in \eqref{u0} in
terms of an initial vorticity $\omega_{0}\in C^\infty_c(\RR^2)$). 
Let $\unuep$ be the solution of  the Navier-Stokes problem \eqref{eq:NSE} on 
$[0,\infty)\times \Omega^\ep$ (where $\Omega^\ep$ is defined in 
\eqref{Kijeps}-\eqref{OM} with $d^\ep \geq \ep$) with initial velocity 
$u_{0}^{\nu,\ep} \in L^\infty\cap L^{2,\infty}(\Omega^\ep)$, satisfying the
divergence-free condition and
no-penetration boundary condition, and such that
$u_{0}^{\nu,\ep}-u_{0}\in L^2(\Omega^\ep)$.
Then, there exists a constant $A$, 
depending only on $\K$, such that if 
\begin{equation}\label{compatibility}
    \frac{\ep}{d_\ep^{(1+\mu)/2}} \leq \frac{A \nu}{\|\omega_{0}\|_{L^1\cap L^\infty}},
\end{equation}
then for any $0<T<\infty$, 
\begin{equation} \label{eq:L2conv}
         \sup_{0\leq t\leq T}  
  \|\unuep-\uE\|_{L^2(\Omega^\ep)}\leq B_{T} \left( \frac{\sqrt\nu}{d_{\ep}^{(1+\mu)/2}}+ \|\unuep_{0}- u_{0} \|_{L^2(\Omega^\ep)} \right),
\end{equation}
where $B_{T}$ is a constant depending only on $T$, $\|\omega_{0}\|_{L^1\cap
W^{1,\infty}}$, and $\K$.
\end{theorem}

We recall that $\Omega^\ep$ depends on the configuration of the obstacles, 
which are arranged according to the parameter $\mu\in [0,1]$.

\begin{remark}
Before giving the proof of the theorem, we list two immediate corollaries:
\begin{enumerate}
\renewcommand{\labelenumi}{{\em (\alph{enumi})}}
 \item By choosing as initial data for Navier-Stokes
 $u_{0}^{\nu,\ep}=u_{0}^\ep$, where $u_0^\ep$ solves \eqref{u0eps},
 Proposition~\ref{prop:u0eps} and Theorem~\ref{th:main} immediately give that
\begin{equation*} 
         \sup_{0\leq t\leq T}  
  \|\unuep-\uE\|_{L^2(\Omega^\ep)}\leq B_{T} 
 \left(
\frac{\sqrt\nu}{d_{\ep}^{(1+\mu)/2}}+ 
\frac{\ep |\ln \ep | }{d_{\ep}^{(1+\mu)/2}} \right),
\end{equation*}
where 
\[
\ep |\ln \ep |  \leq C \sqrt{\ep}  \leq C \sqrt{\nu}.
\]
This estimate then implies Theorem~\ref{th:main1}.
\item In the case of an arbitrary, but fixed and independent of $\ep$, number of
 shrinking obstacles, we recover the rate of convergence of $\sqrt{\nu}$ 
 obtained for one obstacle in \cite{ILN09}.
\end{enumerate}
\end{remark}

\begin{proof}[Proof of Theorem~\ref{th:main}]
With $\uep$ defined in Proposition~\ref{prop:uepest}, we introduce
\[
      \Wnuep := \unuep -\uep,
\]
which is divergence free and verifies the Dirichlet boundary condition on $\partial \Omega^\ep$. We observe that $\uep$ satisfies:
\begin{equation*}
  \uep_t = \phiep\, \uE_t - \hep_t =
   \phiep ( -\uE \nabla \uE - \nabla p^E) - \hep_t.
\end{equation*}
Hence, $\Wnuep$ satisfies:
\begin{equation*}
\begin{split}
  \Wnuep_t - \nu \Delta \Wnuep=& -\nabla p^{\nu,\ep}+ \nu \Delta \uep- (\unuep \cdot \nabla) \unuep \\
  &+\phiep [ (\uE \cdot \nabla) \uE + \nabla p^E ] + \hep_t. 
\end{split}\end{equation*}
Multiplying this equation by $\Wnuep$ and integrating by parts, we obtain:
\begin{equation} \label{eq:enerest}
\begin{aligned}
   \frac{1}{2} \frac{d}{dt} \|\Wnuep  & \|^2_{L^2(\OMep)} +\nu 
\|\nabla \Wnuep\|^2_{L^2(\OMep)} \\
=& -\nu \int_{\OMep} \nabla \Wnuep : \nabla 
\uep \, dx  
 - \int_{\OMep} \Wnuep \cdot \big[ (\unuep \cdot \nabla) \unuep\big] \, dx \\
&+ \int_{\OMep} \Wnuep \cdot \big[ \phiep (\uE \cdot \nabla) \uE \big] \, dx 
 +
\int_{\OMep} \Wnuep \cdot \phiep \nabla p^E\\
& + \int_{\OMep} \Wnuep \cdot \hep_t\, dx.
\end{aligned}
\end{equation}
Under the assumption on the initial data, $\unuep$ and $\uep$ have enough regularity to justify the integration by parts and, hence, the energy identity.

We group the trilinear terms in \eqref{eq:enerest} together and tackle each of the other terms separately. 
We start with the trilinear terms:
\[
  \mathcal{I} := - \int_{\OMep} \Wnuep \cdot \big[
   \big(  (\unuep \cdot \nabla) \unuep \big) -
   \big( \phiep (\uE \cdot \nabla) \uE\big)\big]\, dx,
\] 
which we rewrite, owing to $\unuep=\Wnuep+\uep$ and the divergence-free condition on $\Wnuep$, as
 \begin{align*}
  \mathcal{I} &= - \int_{\OMep} \Wnuep \cdot 
   \big[  (\Wnuep \cdot \nabla) \uep \big] \, dx
  - \int_{\OMep} \Wnuep \cdot 
   \big[  (\uep \cdot \nabla) \uep \big] \, dx \\
   & \qquad\qquad+ \int_{\OMep}  \Wnuep \cdot  \big[ \phiep (\uE \cdot 
   \nabla) \uE \big]\, dx.
 \end{align*}
We will utilize that $\uep$ is close to both $\phiep\, \uE$ and $\uE$ to control the last two terms. Hence, we add and subtract \ $\int_{\OMep} \Wnuep\cdot\big[ (\uep\cdot \nabla) \uE\big]\,dx$:
 \begin{align}
      \mathcal{I} &= - \int_{\OMep} \Wnuep \cdot 
   \big[  (\Wnuep \cdot \nabla) \uep \big] \, dx
  - \int_{\OMep} \Wnuep \cdot 
   \big[  (\uep \cdot \nabla) (\uep-\uE) \big] \, dx 
   \nonumber \\
   &\qquad \qquad + \int_{\OMep}  \Wnuep\cdot \big[ 
   \big((\phiep \uE - \uep)\cdot 
   \nabla\big) \uE \big]\, dx  
   \nonumber \\
    & =  - \int_{\OMep} \Wnuep \cdot 
   \big[  (\Wnuep \cdot \nabla) \uep \big] \, dx
  + \int_{\OMep} (\uep-\uE) \cdot 
   \big[  (\uep \cdot \nabla) \Wnuep \big] \, dx 
   \nonumber \\
   &\qquad \qquad + \int_{\OMep} \Wnuep \cdot  \big[ 
   ( \hep\cdot 
   \nabla) \uE \big]\, dx, \label{eq:trilinear}
 \end{align}
where we have integrated by parts in the second term and have used that $\phiep \uE- \uep = \hep$.

We bound each of the three integrals in \eqref{eq:trilinear} above separately. We set:
\[
     \mathcal{I}_1 := - \int_{\OMep} \Wnuep \cdot 
   \big[  (\Wnuep \cdot \nabla) \uep \big] \, dx.
\]
We have
\[
   \nabla \uep = \phiep \nabla \uE + \nabla \phiep\otimes \uE -\nabla \hep, 
\]
with the last two terms supported on $\Aep$  (defined in \eqref{def Aep}).
Hence, using estimates  \eqref{eq:phiepest} on $\nabla \phi^\ep$,
\eqref{uE infty}-\eqref{grad uE} on $\uE$,  and  
estimate \eqref{i:uepest.c} in Proposition~\ref{prop:uepest} on $\nabla \hep$,
we have  thanks to Lemma~\ref{lem:ladyz}:
 \begin{align}
 | \mathcal{I}_1| \leq&  \|\Wnuep\|^2_{L^2(\OMep)} \| \nabla \uE \|_{L^\infty}  + \|\Wnuep\|^2_{L^4(\Aep)} \|  \uE \|_{L^\infty}  \| \nabla \phi^\ep \|_{L^2}  \nonumber\\
 &+  \|\Wnuep\|^2_{L^4(\Aep)}   \| \nabla \hep \|_{L^2} \nonumber\\
 \leq& C_{0} e^{C_{0} t} \|\Wnuep\|^2_{L^2(\OMep)} + K_{3} 
\frac{\ep}{d_{\varepsilon}^{(1+\mu)/2}}  
\|\nabla\Wnuep\|^2_{L^2(\OMep)}, \label{eq:I1est} 
  \end{align}
where $K_{3}$ depends only on $\omega_{0}$ and $\K$.
From the discussion in Remark \ref{rem.KE}, $K_{3}$ can be taken of the form 
$K_{3}=\tilde{K_{3}} \| \omega_{0}\|_{L^1\cap L^\infty}$, where $\tilde{K_{3}}$
depends only on $\K$.
For
\[
     \mathcal{I}_2 := \int_{\OMep} (\uep- \uE ) \cdot 
   \big[  (\uep \cdot \nabla) \Wnuep \big] \, dx,
\]
we note that ${\rm supp}(\uep- \uE )\subset \Aep$. Then, H\"older's 
and Cauchy's inequalities, together with estimates 
\eqref{eq:phiepest} on $1- \phi^\ep$,  \eqref{uE infty} on $\uE$, and estimates 
\eqref{i:uepest.a} and \eqref{i:uepest.d} in Proposition~\ref{prop:uepest}, 
give that
\begin{align}
  |\mathcal{I}_2| &\leq \|\nabla\Wnuep
  \|_{L^2(\OMep)} \|\uep -\uE \|_{L^4} 
  \|\uep\|_{L^{4}(\Aep)} \nonumber \\
  &\leq \|\nabla\Wnuep
  \|_{L^2(\OMep)}K_E  \frac{\sqrt{\ep}}{   d_{\varepsilon}^{(1+\mu)/4} }\Bigl( \|\uE\|_{L^\infty} \frac{C_{4} \sqrt{\ep}}{   d_{\varepsilon}^{(1+\mu)/4} } + K_E \frac{\sqrt{\ep}}{   d_{\varepsilon}^{(1+\mu)/4} } \Big) \nonumber\\
  &\leq  \frac{\nu}{8}  \|\nabla\Wnuep\|^2_{L^2(\OMep)} + K_{4}  \frac{\ep^2}{ \nu  d_{\varepsilon}^{1+\mu} },  \label{eq:I2est}
\end{align}
where $K_{4}$ depends only on $\omega_{0}$ and $\K$.

Lastly, we bound
\[
     \mathcal{I}_3 :=\int_{\OMep} \Wnuep \cdot  \big[ 
   ( \hep\cdot 
   \nabla) \uE \big]\, dx.
\]
Using that \ ${\rm supp\ }\hep\subset \Aep$, we apply H\"older's and Cauchy's 
inequalities again, together with estimates  \eqref{i:uepest.a} of 
Proposition~\ref{prop:uepest} for $\hep$,  \eqref{uE CZ} for $\nabla \uE$, and 
Lemma~\ref{lem:ladyz}, to obtain:
\begin{align}
   |\mathcal{I}_3| &\leq \|\Wnuep\|_{L^4(\Aep)} 
  \|\hep\|_{L^4} \|\nabla \uE\|_{L^{2}}
  \leq K  \frac{\ep}{   d_{\varepsilon}^{(1+\mu)/4} } \| \nabla \Wnuep\|_{L^2(\OMep)} 
   \nonumber \\
 &\leq  \frac{\nu}{8}  \|\nabla\Wnuep\|^2_{L^2(\OMep)} + K_{5}  \frac{\ep^2}{ \nu  d_{\varepsilon}^{(1+\mu)/2} },  \label{eq:I3est}
\end{align}
where $K_{5}$ depends only on $\omega_{0}$ and $\K$.

From \eqref{eq:trilinear} and \eqref{eq:I1est}, 
\eqref{eq:I2est}, \eqref{eq:I3est}, it follows that
\begin{equation} \label{eq:Iest}
 \begin{aligned}
  |\mathcal{I}| &\leq  \Big(\frac{\nu}{4} +K_{3} \frac{\ep}{d_{\varepsilon}^{(1+\mu)/2}} \Big) \|  \nabla\Wnuep\|^2_{L^2(\OMep)} 
  + C_{0} e^{C_{0} t} \|\Wnuep\|^2_{L^2(\OMep)}  \\
  &\qquad + K_{6}  \frac{\ep^2}{ \nu  d_{\varepsilon}^{1+\mu} },
 \end{aligned}
\end{equation}
 where $K_{6}$ depends only on $\omega_{0}$ and $\K$.

We now turn to the remaining integral terms in \eqref{eq:enerest}.
We begin with
\[
    \mathcal{J} := - \nu \int_{\OMep} \nabla \Wnuep   
      : \nabla \uep \, dx. 
 \]
Applying Cauchy-Schwartz followed by Cauchy's inequality, owing to estimate
(\ref{i:uepest.c}) in Proposition~\ref{prop:uepest} for $\nabla \hep$, 
\eqref{eq:phiepest} for $\nabla \phi^\ep$, and \eqref{uE infty}-\eqref{uE CZ} 
for $\uE$, gives:
\begin{align} 
 |\mathcal{J}| &\leq \sqrt{\nu/2} \|\nabla 
 \Wnuep\|_{L^2(\OMep)} \, \sqrt{2\nu} \|\nabla   
 \uep\|_{L^2} \nonumber\\
 &\leq \frac{\nu}{4}   \|\nabla 
 \Wnuep\|_{L^2(\OMep)}^2 + \nu \Big( \| \uE \|_{L^\infty}  \frac{C_{2}}{d_{\varepsilon}^{(1+\mu)/2}}
 + \| \nabla \uE \|_{L^2}
 +  \frac{K_E}{d_{\varepsilon}^{(1+\mu)/2}} \Big)^2\nonumber\\
 &\leq \frac{\nu}{4}   \|\nabla 
 \Wnuep\|_{L^2(\OMep)}^2 +  K_{7}   \frac{\nu}{d_{\varepsilon}^{1+\mu}},  
\label{eq:Jest}
\end{align}
 where $K_{7}$ depends only on $\omega_{0}$ and $\K$.
 
Finally, we consider the last two integral terms in \eqref{eq:enerest}.
We set:
\[
    \mathcal{H}_1 := \int_{\OMep} \phiep \Wnuep \cdot  \nabla \pE \, dx = - 
\int_{\OMep} \pE \Wnuep \cdot \nabla\phiep \, dx. 
\]
Consequently, by H\"older's and Cauchy's inequalities, and the Poincar\'e's 
inequality \eqref{eq:Poincare}, using that $\nabla \phiep$ is supported on 
$\Aep$ together with estimates \eqref{eq:phiepest} for $\nabla \phi^\ep$ and 
\eqref{pE} for $\pE$, we have
\begin{align}
  |\mathcal{H}_1| &\leq \|\Wnuep\|_{L^2(\Aep)}
 \|\nabla \phiep\|_{L^{2}(\Aep)} 
  \|\pE \|_{L^\infty(\Aep)}  \nonumber \\
  &\leq  \ep K_{1} \,\|\nabla \Wnuep\|_{L^2(\Aep)}
  \frac{C_{2}}{d_{\varepsilon}^{(1+\mu)/2}}
  \|\pE\|_{L^\infty(\Aep)}  \nonumber \\
 &\leq \frac{\nu}{8} 
\|\nabla \Wnuep\|_{L^2(\Aep)}^2 +K_{8}\, \frac{\ep^2}{\nu d_{\varepsilon}^{1+\mu}} ,
 \label{eq:H1est}
\end{align}
 where $K_{8}$ depends only on $\omega_{0}$ and $\K$.

We tackle the last integral term 
\[
   \mathcal{H}_2 := -
\int_{\OMep} \Wnuep \cdot \hep_t\, dx.
\]
As $\hep_t$ is supported on $\Aep$, a similar application of the 
Cauchy-Schwartz, Cauchy's inequalities, and the Poincar\'e's 
inequality \eqref{eq:Poincare}, together with estimate
\eqref{i:uepest.b} in Proposition~\ref{prop:uepest} on $\hep_t$, gives:
\begin{align}
  |\mathcal{H}_2| &\leq \|\Wnuep\|_{L^2(\Aep)}  
  \|\hep_t\|_{L^2(\Aep)} \leq 
   \ep K_{1} \,\|\nabla \Wnuep\|_{L^2(\Aep)} K_E \frac{\sqrt{\ep}}{   d_{\varepsilon}^{(1+\mu)/4} } 
  \nonumber \\
  &\leq \frac{\nu}{8} 
\|\nabla \Wnuep\|_{L^2(\Aep)}^2 +K_{9}\, \frac{\ep^3}{\nu d_{\varepsilon}^{(1+\mu)/2}} ,  \label{eq:H2est}
\end{align}
 where $K_{9}$ depends only on $\omega_{0}$ and $\K$.

Putting together the bounds  \eqref{eq:Iest} obtained for $\mathcal{I}$, 
 \eqref{eq:Jest} for $\mathcal{J}$, and 
\eqref{eq:H1est}-\eqref{eq:H2est} for $\mathcal{H}$, yields the following 
{\em a priori}  energy estimate:
\begin{align*} 
  \frac{1}{2} \frac{d}{dt} \|\Wnuep&\|^2_{L^2(\OMep)} +\nu 
\|\nabla \Wnuep\|^2_{L^2(\OMep)}\\
\leq &
\Big(\frac{3\nu}{4} +K_{3} \frac{\ep}{d_{\varepsilon}^{(1+\mu)/2}} \Big) \|  \nabla\Wnuep\|^2_{L^2(\OMep)} 
  + C_{0} e^{C_{0}t} \|\Wnuep\|^2_{L^2(\OMep)}  \\
  &\qquad +  \frac{\ep^2}{ \nu  d_{\varepsilon}^{1+\mu} } \Big( K_{6}  +K_{8}+ K_{9}\, \ep^3 d_{\varepsilon}^{(1+\mu)/2} \Big) + K_{7}   \frac{\nu}{   d_{\varepsilon}^{1+\mu} }.
\end{align*}

We will apply  Gr\"onwall's Lemma to the previous inequality. We therefore
need the coefficient of $\|\nabla\Wnuep\|_{L^2}^2$ to be strictly less than
$\nu$, which forces the following relation between $\nu$, $\ep$ and $d_{\ep}$:
\begin{equation}\label{constraint}
   \frac{\ep}{d_{\varepsilon}^{(1+\mu)/2}} \leq \tilde A\, \nu, 
\end{equation}
where $A = 1/(4K_{3})$.
Recalling that $K_{3}=\tilde{K_{3}} \| \omega_{0}\|_{L^1\cap L^\infty}$, we
obtain 
the compatibility condition \eqref{compatibility}
with $A = 1/(4\tilde K_{3})$. This condition also implies that
\[
\frac{\ep^2}{ \nu  d_{\varepsilon}^{1+\mu} } \Big( K_{6}  +K_{8}+ K_{9}\, \ep^3 
d_{\varepsilon}^{(1+\mu)/2} \Big) \leq K_{10} \nu,
\]
 where $K_{10}$ depends only on $\omega_{0}$ and $\K$.

For any given $0<T<\infty$, we next define
\[
K_{T}:= C_{0} e^{C_{0} T},
\]
with $C_0$ as in \eqref{grad uE}.
Then,  by Gr\"onwall's Lemma we have that, for any $t\in [0,T]$,
\[
\|\Wnuep(t,\cdot) \|^2_{L^2(\OMep)} \leq  \Big[ \|\Wnuep(0,\cdot) 
\|^2_{L^2(\OMep)} + \frac{K_{10} \nu}{K_{T}}  +   \frac{K_{7}\nu}{ K_{T}  
d_{\varepsilon}^{1+\mu} } 
\Big] e^{2 K_{T}t},
\]
which implies
\begin{equation} \label{eq:conv1}
 \|\Wnuep(t,\cdot) \|_{L^2(\OMep)} \leq  \Bigg[ \|\Wnuep(0,\cdot)
 \|_{L^2(\OMep)} + \sqrt{\frac{K_{10} \nu}{K_{T}}}  + \sqrt{   \frac{ K_{7}
  \nu}{ K_{T}  d_{\varepsilon}^{1+\mu} }}  \Bigg] e^{ K_{T}t}.
\end{equation}

To conclude the proof, we apply the triangle inequality to obtain
\[
 \|(\unuep -\uE)(t,\cdot) \|_{L^2(\OMep)}  \leq
  \|\Wnuep(t,\cdot) \|_{L^2(\OMep)} + \| (\uep-\uE)(t,\cdot)\|_{L^2(\OMep)},
\]
and
\[
 \|\Wnuep(0,\cdot) \|_{L^2(\OMep)}   \leq \|\unuep_{0} -u_{0} \|_{L^2(\OMep)}
 +\| \uep(0,\cdot)- u_{0}\|_{L^2(\OMep)}. 
\]
It then follows from \eqref{eq:phiepest}, \eqref{uE infty}, and point
\eqref{i:uepest.a} in Proposition~\ref{prop:uepest} that
\begin{align}
  \| (\uep-\uE)(t,\cdot)\|_{L^2(\OMep)}& \leq \| \uE \|_{L^\infty}
\|1-\phi^\ep\|_{L^2} + \| \hep \|_{L^4} \| 1 \|_{L^4(\Aep)} \nonumber \\
  & \leq C \frac{\ep}{d_{\ep}^{(1+\mu)/2}} \leq C' \nu,  \label{eq:conv2}
\end{align}
for all $t\in [0,\infty)$, using also condition \eqref{constraint}.
Lastly, we combine \eqref{eq:conv1}-\eqref{eq:conv2} and note that
\[
   \nu, \sqrt\nu \leq C \frac{\sqrt\nu}{d_{\ep}^{(1+\mu)/2}},
\]
for some $C>0$, as $\nu$ and $d_{\ep}$ are both bounded above.
Estimate \eqref{eq:L2conv} then follows.

\end{proof}

\begin{remark}\label{rem:end}
The dependence on $d_\ep$ in estimate \eqref{eq:L2conv} comes from the 
contribution of $\mathcal{J}$, more precisely from the terms:
\[
\nu \int_{\Aep} \nabla \Wnuep  : [    \nabla \phi^\ep\otimes \uE] \, dx
\quad \text{and}\quad
\nu \int_{\Aep} \nabla \Wnuep  : \nabla \hep \, dx,
\]
Without these terms, we would establish a rate of convergence of order  
$\sqrt{\nu}$ instead to $\frac{\sqrt{\nu}}{d_{\ep}^{(1+\mu)/2}}$, as in 
the case of a fixed number of inclusions. Even though these integrals are 
taken on $\Aep$, the measure of which is small, the $L^2$-norm of $\nabla 
\phi^\ep$ is unbounded \eqref{eq:phiepest}, and hence we are not able to prove 
these terms are of the same order as the other terms in the energy estimate.

To improve the energy estimate, we can ask whether there exists a better choice
for $\phi^\varepsilon$ than our construction \eqref{eq:phi}-\eqref{phi ep}.
Namely, we look for a $\tilde\phi^\varepsilon$ that  vanishes in
$B(0,d_{\varepsilon})^c$, is equal to $1$ in $B(0,\varepsilon)$, and  such that
$\| \nabla \tilde\phi^\varepsilon\|_{L^2}$ is the smallest possible. This
optimization problem has a unique solution, which is the solution of
\[
\Delta \tilde\phi^\varepsilon = 0,\quad \tilde\phi^\varepsilon\vert_{\partial
B(0,\varepsilon)}=1, \quad \tilde\phi^\varepsilon\vert_{\partial
B(0,d_{\varepsilon})}=0.
\]
This slution has an explicit expression of the form: 
\[
\tilde\phi^\varepsilon(x)= \frac{\ln ( \frac{|x|}{d_{\varepsilon}}) }{\ln ( \frac{\varepsilon}{d_{\varepsilon}})}.
\]

Defining
\begin{equation}\label{new phi}
 \phi^\varepsilon(x) := 1- \sum_{i=1}^{n_{1}^\varepsilon} 
\sum_{j=1}^{n_{2}^\varepsilon} \tilde \phi^\varepsilon
\left(x-z_{i,j}^\ep\right),
\end{equation}
it follows for $d_{\varepsilon} \gg \varepsilon$,
\[
\| \nabla \phi^\varepsilon\|_{L^2} \leq C\sqrt{\frac{n_{1}^\varepsilon n_{2}^\varepsilon}{|\ln \varepsilon|}}\leq \frac{C}{d_{\ep}^{(1+\mu)/2} \sqrt{|\ln \varepsilon|}},
\]
which is a slight improvement over \eqref{eq:phiepest}. We cannot expect a rate
 better than $\frac{\sqrt{\nu}}{d_{\ep}^{(1+\mu)/2} \sqrt{|\ln
\varepsilon|}}$ in \eqref{eq:L2conv}, because, at $\nu$ fixed, it is well known
that a ``strange term'' appears when
\[
d_{\ep}^{(1+\mu)/2} \sqrt{|\ln \varepsilon|} \to C >0.
\]
(See \cite{Allaire90a,CM82} for the case $\mu =1$ and \cite{Allaire90b} for the
case $\mu=0$.)

We  note that for $\mu =0$, the situation in which the distance is smaller than
the critical distance (namely, $d_{\ep}^{1/2} \sqrt{|\ln
\varepsilon|} \to 0$) was not treated in \cite{Allaire90b} at fixed viscosity.
This case was considered only when the inclusions are distributed in both
directions (related to the case  $\mu=1$) filling a bounded domain
\cite{Allaire90a}. Indeed, this set up implies a uniform Poincar\'e estimate.
Allaire investigated the limit of
$u^{\nu,\varepsilon}/(d_{\ep} \sqrt{|\ln \varepsilon|})$, which in turn implies
a limit for the term:
\[
\int_{\Aep} |\nabla  u^{\nu,\varepsilon}| |    \nabla \phi^\ep | \, dx = \int_{\Aep} \Bigg|\frac{ \nabla  u^{\nu,\varepsilon}}{d_{\ep} \sqrt{|\ln \varepsilon|}} \Bigg| \Big|   d_{\ep} \sqrt{|\ln \varepsilon|} \nabla \phi^\ep \Big| \, dx.
\]

Therefore, guided by the optimal results for the Laplace and Stokes equations,
it would be natural to consider $\phi^\varepsilon$ defined as in \eqref{new phi}
rather than in \eqref{phi ep}. However, the more natural construction of the
cut-off function leads to a series of technical complications.
Besides the adaptation of \eqref{eq:phiepest}, we have used several times 
elliptic estimates on a domain of the form $U=(-2,2)^2\setminus \K$ and the fact
that $\Aep=\cup (z_{i,j}^\ep+\ep U)$. On the other hand, dilating 
$B(0,d_{\varepsilon})\setminus B(0,\varepsilon)$ by a factor of $\ep$ results
in a set the size of which grows unboundedly as $\varepsilon\to 0$. To adapt
the approach of Section \ref{sec:prelim} to the improved setting, we would
therefore have to:
\begin{itemize}
 \item estimate the Bogovski{\u\i}  constant $\widetilde{C}_{p}$ in Lemma \ref{lem:hep} in terms of $\varepsilon$, which should be uniformly bounded thanks to \cite[Theorem III.3.1]{Galdi94};
 \item estimate the Poincar\'e constant $K_{1}$ in Lemma \ref{lem:poincare};
 \item estimate the embedding constant $K$ in the proof of Lemma \ref{lem:ladyz}.
\end{itemize}
\end{remark}

\end{document}